\documentclass[a4paper,10pt]{article}
\usepackage[utf8]{inputenc}
\usepackage{amsmath,amsthm}
\usepackage{amsfonts}  
\usepackage{amssymb}
\usepackage{graphicx} 
\usepackage{bm,color}

\def\minus{%
{\resizebox{ 2px}{!}{\rm{-}}}
}
\def\plus{
{\resizebox{4px}{!}{\bf{+}}}
}

 \newtheorem{theorem}{Theorem}[section]
 \newtheorem{lemma}[theorem]{Lemma}
 \newtheorem{corollary}[theorem]{Corollary}
 \newtheorem{definition}[theorem]{Definition}

 \newenvironment{remark}{ \refstepcounter{theorem}{\bf Remark \thetheorem~}}{\hfill$\square$\newline\smallskip\newline}
 
 \newenvironment{example}{ \refstepcounter{theorem}{\bf Example \thetheorem~}}{\hfill$\square$\newline\smallskip\newline}

\newcommand\samethanks[1][\value{footnote}]{\footnotemark[#1]}

\title{Trace operator on $H^1(\Omega)$ for general open bounded domains}

\author{R. Eymard\thanks{Universit\'e Gustave Eiffel, LAMA, (UMR 8050), UPEM, UPEC, CNRS, F-77454, Marne-la-Vallée (France)\newline
{\tt robert.eymard, david.maltese, yannick.vincent@univ-eiffel.fr}}, T. Gallou\"et\thanks{I2M UMR 7373, Aix-Marseille Universit\'e, CNRS, Ecole Centrale de Marseille,
F-13453 Marseille (France)\newline
{\tt thierry.gallouet@anciens.univ-amu.fr}}, D. Maltese\samethanks[1] and Y. Vincent\samethanks[1]}

\begin{document}

\maketitle

\begin{abstract} In the case of any bounded open set $\Omega\subset{\mathbb R}^d$ with boundary $\partial\Omega$, we first construct a directional trace in any direction $\theta$ of the unit sphere, for any  $u\in L^2(\Omega)$ whose the directional derivative $\partial_\theta u$ in the direction $\theta$ belongs to $L^2(\Omega)$. This directional trace is shown to belong to $L^2(\partial\Omega,\mu_\theta)$, where $\mu_\theta$ is a measure supported by the closure of all points of $\partial\Omega$ which are the extremity of an open segment directed by $\theta$, included in $\Omega$. This trace enables an integration by parts formula. We then show that the set $H_{\rm tr}^1(\Omega)$ containing the elements of $H^1(\Omega)$ whose the directional trace does not depend on $\theta$  is closed. It therefore contains the closure of $H^1(\Omega)\cap C^0(\overline{\Omega})$ in  $H^1(\Omega)$. Examples where $H_{\rm tr}^1(\Omega) = H^1(\Omega)$ and $H_{\rm tr}^1(\Omega) \neq H^1(\Omega)$ are provided.
 
\end{abstract}

\section{Introduction}

Let $d$ be a strictly positive integer and let $\Omega$ be a non-empty bounded open subset of ${\mathbb R}^d$, with boundary $\partial\Omega = \overline{\Omega}\setminus\Omega$. We focus in this paper on the notion of trace on $\partial\Omega$ for the elements of $H^1(\Omega)$, in connection with the existence of an integration by parts formula. 

\medskip

Let us recall that the notion of trace is standardly introduced in the case of domain with Lipschitz regularity. The existence of this trace is obtained through regular local charts which transform the computations at the boundary of the domain into computations on the boundary in ${\mathbb R}^d$ of the set ${\mathbb R}^+\times {\mathbb R}^{d-1}$ (see the seminal paper \cite{Gagliardo1957}, and see \cite[Theory of traces]{brezis} for the introduction of local charts). Then the trace of any function $u\in H^1(\Omega)$ can be defined by passing to the limit in the space $L^2(\partial\Omega,\mathcal{H}^{d-1})$, where $\mathcal{H}^{d-1}$ denotes the $(d-1)$-dimensional Hausdorff measure on $\mathbb{R}^d$ (see \cite[Definition 2.1]{evans2015meas}), on the restriction to $\partial\Omega$ of continuous functions converging to $u$ in $H^1(\Omega)$.

\medskip

Hence, for more general domains, it is natural to extend this procedure, by defining the trace at the boundary $\partial\Omega$ for the elements of the closure  of $H^1(\Omega)\cap C^0(\overline{\Omega})$ in $H^1(\Omega)$, denoted by $\widetilde{H}^1(\Omega)$. This yields the following definition of the approximative trace \cite{ae2011dirtoneu,are2024perron,are2003lap,sauter2020uniq}:
 \begin{definition}[Approximative trace]\label{def:aptrace} Let  $\Omega$ be a bounded open subset of ${\mathbb R}^d$ and let $u \in H^1(\Omega)$. We say that $\varphi \in L^2(\partial \Omega, \mathcal{H}^{d-1})$ is an approximative trace of $u$ if there exists a sequence $(u_n)_{n \ge 0}$ in $H^1(\Omega) \cap C^0(\overline{\Omega})$ such that the sequence $(u_n)_{n \ge 0}$ converges to $u$ in $H^1(\Omega)$ and the sequence 
 $(u_n)_{n \ge 0}$ converges to $\varphi$ in $L^2(\partial \Omega,\mathcal{H}^{d-1})$.
 \end{definition}
 This definition is used in \cite{ae2011dirtoneu} with the objective to study the Dirichlet-to-Neumann operator, under the hypothesis that $\mathcal{H}^{d-1}(\partial\Omega)<+\infty$ and that the only trace of $0$ is $0$ (this last condition is sufficient and necessary for the uniqueness of the approximative trace). This uniqueness property is discussed in \cite{are2024perron} and is intensively studied in \cite{sauter2020uniq}, considering functions on the boundary only locally integrable for the measure $\mathcal{H}^{d-1}$ (counter examples for this uniqueness are given in \cite[Example 4.3]{are2003lap} or \cite[end of Section 3]{bucur2010var}). As noticed in these works, this uniqueness property is related to the fact that the boundary of the domain may include points which are not connected with $\Omega$, and are therefore not supporting the trace of the elements of $H^1(\Omega)$.

 \medskip
 
 Similar observations are done in \cite{shvartsman}, in order to define the trace of functions of Sobolev spaces $W^{1,p}(\Omega)$ with $p>d$ on the boundary of general domains, with introducing a notion of ``$\alpha$-accessible points''. Following the same kind of ideas, we define and study a measure  on the boundary, denoted by $\mu_\theta$, depending on the direction $\theta\in \mathcal{S}$, where $\mathcal{S}$ is the unit sphere of  ${\mathbb R}^d$, whose the support is determined by the set $\partial_\theta\Omega$, containing the points of $\partial \Omega$ which are the extremity of open segments directed by $\theta$ and included in $\Omega$ (see Definition \ref{def:dirmeasurebound} of a directional measure). A series of properties of this measure is proved in Section \ref{sec:openbounded}, including an important one concerning the negligible sets for this measure (see Lemma \ref{lem:negl} and corollary \ref{cor:negl2}). The remaining of this paper is then devoted to the definition and the study of trace operators which are integrable for the measures $(\mu_\theta)_{\theta\in \mathcal{S}}$.

\medskip
 
A first step in this direction is the definition, in Section \ref{sec:traceund}, of a trace operator called the directional trace, whose domain is the space $W_\theta(\Omega)$ of all functions of $L^2(\Omega)$ whose directional derivative $\partial_\theta u$ also belongs to  $L^2(\Omega)$. For fixing ideas, let us suppose that $\theta=(1,0,\ldots,0)$, and then that $\partial_\theta u$ is the derivative with respect to the first coordinate. Lemma \ref{lem:underivative} in the appendix shows that the elements $v\in W_\theta(\Omega)$ are such that $v(\cdot,y)\in H^1(\omega_\theta(y))$ for a.e. $y\in\mathbb{R}^{d-1}$ such that $\omega_\theta(y) = \{s\in\mathbb{R}, (s,y)\in\Omega\}$ is non-empty.
 This lemma is then extending some results proved in \cite[Theorems 1 and 2]{mazya}, or in 
\cite[Theorem 2 p. 164]{evans2015meas} for the whole space, to the case of $W_\theta(\Omega)$ for a general bounded open domain $\Omega$. The directional trace $\gamma_\theta u$ is then defined as the one-dimensional trace of the functions $v(\cdot,y)$ at the points of $\partial_\theta\Omega$. Hence, in the case of regular domains, the directional trace of a function $u\in H^1(\Omega)\subset W_\theta(\Omega)$ coincides with the trace in the sense of  \cite{brezis}.

\medskip
 
We then show in Theorem \ref{thm:deftraceoned} that this directional trace $\gamma_\theta u$ is uniquely defined, up to negligible sets for the measure $\mu_\theta$. Theorem \ref{thm:proptraceoned} states that it belongs to the space $L^2(\partial\Omega,\mu_\theta)$ and that the following integration by parts formula holds for any $u,v\in  W_\theta(\Omega)$:
 \begin{equation}\label{eq:intpartintro}
 \int_\Omega (u(x)\partial_\theta v(x) +v(x)\partial_\theta u(x)){\rm d}x = \int_{\partial\Omega} \frac{\gamma_{\theta} u(z)\gamma_{\theta} v(z) -\gamma_{\minus \theta} u(z_{\minus \theta}(z))\gamma_{\minus \theta} v(z_{\minus \theta}(z))}{|z- z_{\minus\theta}(z)|} {\rm d}\mu_\theta(z),
 \end{equation}
 where the point $z_{{\minus\theta}}(z)$ is the first intersection between $\partial\Omega$ and the half line starting from $z$ with direction ${-\theta}$. 

 \medskip

 We then observe in a second step that the elements of $H^1(\Omega)$ inherit the properties of the elements of $W_\theta(\Omega)$, since $H^1(\Omega) = \bigcap_{\theta\in \mathcal{S}}W_\theta(\Omega)$. Then, for a given $u\in H^1(\Omega)$, the question whether the values of $\gamma_{\theta} u(z)$ at a given point $z\in\partial\Omega$ are or are not depending on $\theta$ arises. Indeed, in the case of regular domains for which there exists a trace for the elements of  $H^1(\Omega)$ in the standard sense, we already noticed that this standard trace is equal to the directional trace in any direction $\theta\in \mathcal{S}$. 
 Indeed,  situations where one can find $\theta\neq \theta'\in \mathcal{S}$, $A\subset\partial\Omega$ with $\mu_\theta(A)>0$ and $\mu_{\theta'}(A)>0$ and $u\in H^1(\Omega)$ with $\gamma_\theta u(z)\neq \gamma_{\theta'} u(z)$ for $\mu_\theta$-a.e. and $\mu_{\theta'}$-a.e. $z\in A$ are easy to provide (see examples \ref{exa:fisund} for a one-dimensional example of a non-connected domain, and  \ref{exa:fisdeuxd} for a two-dimensional example of a connected domain). Note that, even in such situations, the integration by parts formula  \eqref{eq:intpartintro} still holds in $H^1(\Omega)$ using the directional traces.
 
 \medskip

 Since, in the case of a general domain, a unique trace function satisfying integration by parts formula cannot be defined for all elements of $H^1(\Omega)$, we define the set  $H_{\rm tr}^1(\Omega)$ of all $u\in H^1(\Omega)$ such that there exists a unique function ${\rm tr }(u)\in L^2(\partial\Omega,(\mu_\theta)_{\theta\in\mathcal{S}})$, which coincides with  $\gamma_\theta u$ for all $\theta\in \mathcal{S}$. The space $L^2(\partial\Omega,(\mu_\theta)_{\theta\in\mathcal{S}})$ is defined as the equivalence class of functions belonging to $\mathcal{L}^2(\partial\Omega,\mu_\theta)$ whose the norm is bounded independently of $\theta$ (see Definition \ref{def:tracemono}). We show in Theorem \ref{thm:huntclosed} that  $H_{\rm tr}^1(\Omega)$, which is therefore the domain of the operator ${\rm tr }$, is closed in $H^1(\Omega)$. This trace operator can then be used for the study of variational problems in general domains, since $H_{\rm tr}^1(\Omega)$ is a Hilbert space, and since the trace of $u\in H_{\rm tr}^1(\Omega)$, equal to the directional traces by construction, therefore satisfies the  integration by parts formula  \eqref{eq:intpartintro}. Note that, in the case of regular domains (see examples \ref{exa:cun} and \ref{exa:lip}), the right-hand side of  \eqref{eq:intpartintro} can be expressed through the use of the outward unit normal to the boundary ${\bm n}(z)$ and the measure $\mathcal{H}^{d-1}$:
\[
 \int_\Omega (u(x)\partial_\theta v(x) +v(x)\partial_\theta u(x)){\rm d}x = \int_{\partial\Omega} {\rm tr\, } u(z){\rm tr\, }v(z)\,  \theta\cdot {\bm n}(z) {\rm d}\mathcal{H}^{d-1}(z).
 \]
In the case of a general domain (see Example \ref{exa:cuspidal} of a cuspidal domain), it may happen that the measure $\mathcal{H}^{d-1}(z)$ be defined a.e., but the function $z\mapsto {\rm tr\, } u(z){\rm tr\, }v(z)\,  \theta\cdot {\bm n}(z) $ be not integrable (such questions are discussed in \cite{poborchi}). We nevertheless prove in this paper that, even in such cases, the combination in the right-hand side of  \eqref{eq:intpartintro} between the values of the traces of $u$ and $v$ at the two boundary points $z$ and $z_{\minus\theta}(z)$ always provides an integrable function for the measure $\mu_\theta$. 

\medskip

A consequence of the fact that $H_{\rm tr}^1(\Omega)$ is closed in $H^1(\Omega)$, is that it contains the closure $\widetilde{H}^1(\Omega)$ of $H^1(\Omega)\cap C^0(\overline{\Omega})$, which proves that the limit in $L^2(\partial\Omega,(\mu_\theta)_{\theta\in\mathcal{S}})$ of the restriction to the boundary of the elements of  $H^1(\Omega)\cap C^0(\overline{\Omega})$ is unique, contrarily to the approximative trace in the sense of Definition \ref{def:aptrace}. We present in this paper a few examples where the notion of approximative trace in the sense of Definition \ref{def:aptrace} cannot be used, or leads to uniqueness problems.

As a consequence of $\widetilde{H}^1(\Omega)\subset H_{\rm tr}^1(\Omega)$, we observe that $H_{\rm tr}^1(\Omega) = H^1(\Omega)$ necessarily holds in all cases where $\widetilde{H}^1(\Omega) = H^1(\Omega)$, for example in the following situations:
\begin{enumerate}
\item In the case where $\Omega$ is a Lipschitz domain, which implies that the trace constructed in this paper coincides with the standard trace;
 \item In the more general case where $\Omega$ has a continuous boundary (in the sense of the existence of continuous local charts) \cite[Theorem 2 p.11]{mazya};
 \item In the case where $\Omega$ is bounded by a Jordan curve \cite[Theorem 1 p. 256]{lewis}. This includes for example the case of the Koch snowflake domain \cite{kaz2024koch}.
\end{enumerate}

\medskip

Note that all the results of the present paper can be extended to the case of $W^{1,p}(\Omega)$ for $p\in [1,+\infty)$.
Nevertheless, many open questions remain to be handled:

\begin{itemize}
 \item The properties of domains $\Omega$ such that the trace as defined in this paper and the approximative trace coincide have to be specified.
 \item The link between the collection of measures $(\mu_\theta)_{\theta\in\mathcal{S}}$ and the harmonic measures on $\partial\Omega$ remains to be established in the general case, in the spirit of \cite{dahlberg}.
 \item The existence of situations where the directional traces of elements of $H^1(\Omega)$ in all directions $\theta\in\mathcal{S}$ can take more than two values must be studied.
\end{itemize}

\medskip
\medskip

{\it Notation for the whole paper:} 
\begin{itemize}

\item The dimension of the considered space is a strictly positive integer $d$, and  ${\mathbb R}^d$ is defined as an Euclidean space, with the Euclidean norm denoted by $|x|$ for any $x = (x_1,\ldots,x_d)\in {\mathbb R}^d$.

\item We denote by $\Omega\subset {\mathbb R}^d$ a given non-empty bounded open set, and we denote by $\partial{\Omega}$ its boundary. The diameter of $\Omega$, that is the supremum of the Euclidean distance between two points of $\Omega$, is denoted by ${\rm diam}(\Omega)$. 

\item  We denote by  $\lambda^{d}$ the $d$-dimensional Lebesgue measure in $\mathbb{R}^{d}$, and we simplify the notation ${\rm d}\lambda^{d}(x)$ in ${\rm d}x$, for $x\in \mathbb{R}^{d}$; we denote by $\lambda^{d-1}$ the $(d-1)$-dimensional Lebesgue measure on any hyperplane of $\mathbb{R}^{d}$, identified with $\mathbb{R}^{d-1}$ by any linear isometry.

\item  We denote by $\mathcal{H}^{d-1}$ the $(d-1)$-dimensional Hausdorff measure on $\mathbb{R}^{d}$.

\item  Let $(E,T,m)$ be a measured space ($T$ is a $\sigma$-algebra and $m$ is a measure on $(E,T)$). Any subset of $E$ is $m$-negligible if it is included in an element $A$ of $T$ such that  $m(A)=0$. Two functions $f$, $g:E\to\mathbb{R}$ on $E$ are $m$-a.e. equal if the set $\{x\in E,\ f(x)\neq g(x)\}$ is $m$-negligible.

\item We denote by $\mathcal{L}^p(E,T,m)$ (resp. $L^p(E,T,m)$) the space of all functions from $E\to \mathbb{R}$ (resp. all equivalence class for the relation $m$-a.e. equal) which are $m$-measurable for the $\sigma$-algebra $T$ on $E$ (that is  $m$-a.e. equal to a measurable function), and whose power $p$ is integrable for the measure $m$ defined on $(E,T)$. We omit $T$ when it is the Borel algebra on $E$. 

\item Due to risks of confusion with any ordered pair of values in this paper, we denote the open real intervals $]a,b[$ instead of $(a,b)$.
\end{itemize}

\section{Directional measures on the boundary}\label{sec:openbounded}

Let $\mathcal{S}$ be the unit sphere of ${\mathbb R}^d$ for the Euclidean norm, that is
\[
 \mathcal{S}  = \{x\in {\mathbb R}^d,\ |x| = 1\}.
\]
For any $\theta\in \mathcal{S}$ and any $x\in \Omega$, we denote by
\begin{equation}\label{eq:defb}
 \delta_\theta(x) := \sup\{s\in[0,+\infty), \forall t\in[0,s[,\ x+t\theta\in \Omega\}\hbox{ and }z_\theta(x) := x+\delta_\theta(x)\theta \in\partial\Omega,
\end{equation}
and we also denote by
\begin{equation}\label{eq:defl}
 \ell_\theta(x) := \delta_\theta(x)+\delta_{\minus\theta}(x) = |z_\theta(x) - z_{\minus\theta}(x)|.
\end{equation}
Note that $z_\theta(x), z_{\minus\theta}(x)\in\partial\Omega$ and $\ell_\theta(x)$ are constant on the segment $\{x+s\theta, -\delta_{\minus\theta}(x)<s<\delta_\theta(x)\}$. If we denote by $z\in \partial_\theta\Omega$ the point $z_\theta(x)$ and by $\widehat{z}\in \partial_{\minus\theta}\Omega$  the point $z_{\minus\theta}(x)$, we can uniquely extend by continuity the function $z_{\minus\theta}$ and $ \ell_\theta$ to the point $z$ with setting $ z_{\minus\theta}(z) = \widehat{z}$ and $ \ell_\theta(z) = |z - \widehat{z}|$.

\medskip

For any $x\in \mathbb{R}^d$, we denote by $\mathcal{P}_\theta(x) = x - (x\cdot\theta) \theta$ the orthogonal projection on the hyperplane $H_\theta\subset \mathbb{R}^d$ orthogonal to $\theta$ passing by $0$. 

Let us define the following sets:
 \begin{equation}\label{eq:defomegatheta}
    \forall y\in H_\theta,\ \omega_\theta(y) = \{ s\in \mathbb{R}~\text{such that}~ s\theta+y \in \Omega\}\subset \mathbb{R}.
 \end{equation}
  Note that, for any $y \in \mathcal{P}_\theta(\Omega)$, $\omega_\theta(y)$ is a non-empty open subset of $\mathbb{R}$. We then denote by  $\mathcal{I}_{\theta}(y)$ the non empty countable set of disjoint bounded open intervals such that $\omega_\theta(y) = \bigcup_{]\alpha,\beta[\in \mathcal{I}_{\theta}(y)} ]\alpha,\beta[$.

Let $\partial_\theta\Omega$ be the part of the boundary defined by
\begin{equation}\label{eq:defdthetaomega}
\partial_\theta\Omega :=  z_\theta(\Omega)
\end{equation}
We then denote by
\begin{equation}\label{eq:defdomegatilde}
  \widetilde{\partial\Omega} = \bigcup_{\theta\in{\mathcal{S}}} \partial_\theta \Omega \subset\partial\Omega.
\end{equation}
We have the following properties.

\begin{lemma}\label{lem:partialthetaboundary}
 Let $z \in \partial_\theta \Omega$. Then  $\omega_\theta(\mathcal{P}_\theta  z)$ is a non-empty open subset  of $\mathbb{R}$  and there exists $]\alpha,\beta[ \in \mathcal{I}_{\theta}(\mathcal{P}_\theta  z)$ such that $z=\mathcal{P}_\theta z + \beta \theta$. Therefore
\begin{equation}\label{eq:propdthetaomega}
\partial_\theta\Omega := \{z\in\partial\Omega,\exists r>0,\ \forall t\in (0,r), z-t\theta\in\Omega\} 
=\{ \beta\theta+y, \ y\in \mathcal{P}_\theta(\Omega), \ ]\alpha,\beta[\in \mathcal{I}_{\theta}(y)\},
\end{equation}
and
\begin{equation}\label{eq:eqptheta}
 \mathcal{P}_\theta(\Omega) = \mathcal{P}_\theta(\partial_\theta\Omega).
\end{equation}
\end{lemma}
\begin{proof}
 Let $x\in\Omega$. Then, letting $y = \mathcal{P}_\theta(x)$,  there exist $]\alpha,\beta[\in\mathcal{I}_{\theta}(y)$ and $s\in ]\alpha,\beta[\subset \omega_\theta(y)$ such that $x = s\theta+y$. Then $z := z_\theta(x) = \beta\theta+y \in \partial_\theta\Omega$ and $\mathcal{P}_\theta(z) = y$.
 We then notice that, letting $r=\beta-\alpha$, we have $x = z-t\theta$ with $t = \beta-s\in (0,r)$.
\end{proof}

\begin{lemma}\label{lem:closomegatilde} 
The closure and the boundary of $\widetilde{\partial\Omega}$ are equal to $\partial\Omega$.
\end{lemma}
\begin{proof}
Let $z\in \partial\Omega$, and let,  for all $n\in\mathbb{N}$, $x_n\in\Omega$ such that the sequence $(x_n)_{n\in\mathbb{N}}$ converges to $z$ as $n\to+\infty$.  Hence $x_n\neq z$ for all $n\in\mathbb{N}$. Let $\theta_n\in\mathcal{S}$ such that
\[
 \theta_n = \frac 1 {|z - x_n|}(z - x_n).
\]
Let $z_n = z_{\theta_n}(x_n) = x_n + \frac {\delta_{\theta_n}(x_n)}{|z - x_n|}(z - x_n)$. We have $z_n\in \widetilde{\partial\Omega}$ and
\[
 0 < \frac {\delta_{\theta_n}(x_n)}{|z - x_n|}\le 1,
\]
since $z$ belongs to the half line issued from $x$ with direction $\theta_n$. Therefore
\[
 |z - z_n|\le |z - x_n|,
\]
which proves that $z$ belongs to the closure of $\widetilde{\partial\Omega}$. Hence $\partial\Omega$ is a subset of the closure of $\widetilde{\partial\Omega}$. Since $\partial\Omega$ is closed, this concludes the proof that $\partial\Omega$ is the closure of $\widetilde{\partial\Omega}$. Since the interior of $\partial\Omega$ is empty, we get that $\partial\Omega$ is the boundary of $\widetilde{\partial\Omega}$.
\end{proof}

\begin{lemma}\label{lem:measr}
For  any $\theta\in \mathcal{S}$  the function $\delta_{\theta} : \Omega\to{\mathbb{R}}$  is lower semicontinuous and therefore measurable. 
\end{lemma}
\begin{proof}
Let $x\in\Omega$, and let $\lambda < \delta_{\theta}(x)$. Assume that, for all $n\in\mathbb{N}$ with $B(x,\frac 1 {n+1})\subset\Omega$, there exists $x_n\in B(x,\frac 1 {n+1})$ such that $\delta_\theta(x_n) <\lambda$. It means that, for all $n\in\mathbb{N}$, we can find $t_n\in(0,\lambda)$ with $x_n + t_n\theta\notin\Omega$. Let us consider a subsequence $x_{\varphi(n)}$ for $n\in\mathbb{N}$ such that $t_{\varphi(n)}$ converges to some $\overline{t}\in [0,\lambda]$. Since the complementary of $\Omega$ is closed, and $x_{\varphi(n)}\to x$ as $n\to\infty$, we get that $x + \overline{t}\theta\notin\Omega$. This contradicts 
$0\le \overline{t}\le \lambda < \delta_{\theta}(x)$. Hence there exists $n\in\mathbb{N}$ with $B(x,\frac 1 {n+1})\subset\Omega$ such that, for all $y\in B(x,\frac 1 {n+1})$, it holds $\delta_\theta(y) \ge\lambda$, which concludes the proof of the lemma.
\end{proof}

\begin{remark}
 The preceding proof could be extended to prove that the function defined on $\mathcal{S}\times\Omega$ by $(\theta,x)\mapsto \delta_{\theta}(x)$ is  lower semicontinuous.
\end{remark}

\begin{definition}[Directional measure on the boundary]\label{def:dirmeasurebound} We define the measure $\mu_{\theta}$  on the boundary of $\Omega$ by
\begin{equation}\label{eq:defmutheta}
\mu_{\theta}(A) =  \int_{\Omega} \chi_A( z_\theta(x)) {\rm d}x,~\text{for any}~ A \in {\mathcal B}(\partial \Omega),
\end{equation}
where $\chi_A$ is the characteristic function of $A$ and ${\mathcal B}(\partial \Omega)$ is the set of all Borel sets of $\partial \Omega$
(that is the intersection with $\partial \Omega$ of Borel sets of $\mathbb R^d$).
Then $\mu_{\theta}$ is a finite non-negative $\sigma$-additive function on ${\mathcal B}(\partial \Omega)$ (that is to say a finite measure on ${\mathcal B}(\partial \Omega)$) since it is bounded by $\lambda^d(\Omega)\le {\rm diam}(\Omega)^{d} $.
\end{definition}

\begin{lemma}\label{lem:negl} Let $\theta \in \mathcal{S}$.   
  Let $A$ be a non-empty element of $ {\mathcal B}(\partial_\theta\Omega)$, where ${\mathcal B}(\partial_\theta\Omega)$ is the family of all Borel sets of $\partial_\theta\Omega\subset \partial\Omega$. Then $\mu_\theta(A)>0$ if and only if $\lambda^{d-1}(\mathcal{P}_\theta(A))>0$.

\end{lemma}
\begin{proof}
Let $A$ be a non-empty element of $ {\mathcal B}(\partial_\theta\Omega)$.
We proceed in several steps.

\textbf{Step 1.}
Using the notations of \eqref{eq:defomegatheta} we have
\begin{multline*}
 \int_\Omega \chi_{A}( z_\theta(x)) {\rm d}x  = \int_{H_\theta} \Big( \int_{\omega_\theta(y)} \chi_{A}( z_\theta(s\theta+y))  {\rm d}\lambda_1(s) \Big) {\rm d}\lambda^{d-1}(y) 
\\
=  \int_{\mathcal{P}_\theta A} \Big( \int_{\omega_\theta(y)} \chi_{A}( z_\theta(s\theta+y))  {\rm d}\lambda_1(s) \Big) {\rm d}\lambda^{d-1}(y) +  \int_{H_\theta \setminus \mathcal{P}_\theta A} \Big( \int_{\omega_\theta(y)} \chi_{A}( z_\theta(s\theta+y)) {\rm d}\lambda_1(s) \Big) {\rm d}\lambda^{d-1}(y).
\end{multline*}

\textbf{Step 2.}
Let $y \in H_\theta \setminus \mathcal{P}_\theta A$ such that $\omega_\theta(y)$ is non-empty  and let $s \in \omega_\theta(y)$. We have $z_\theta(s \theta  +y) = (s+ \delta_\theta(s\theta+y))\theta+y$ which gives $\mathcal{P}_\theta z_\theta(s\theta+y) = y \in H_\theta \setminus \mathcal{P}_\theta A$. This implies that  $z_\theta(s \theta+y) \in \partial \Omega \setminus A$ for any $y \in H_\theta \setminus \mathcal{P}_\theta A$ such that $\omega_\theta(y)$ is non-empty and for any $s \in \omega_\theta(y)$. We obtain 
$$
\int_{H_\theta \setminus \mathcal{P}_\theta A} \Big( \int_{\omega_\theta(y)} \chi_{A}( z_\theta(s\theta+y)) {\rm d}\lambda_1(s) \Big) {\rm d}\lambda^{d-1}(y) =0.$$

\textbf{Step 3.}
Let $y \in \mathcal{P}_\theta A$.  We have the following remarks.
\begin{enumerate}
    \item 
There exists $z_y \in A$ such  that
$ y + (z_y\cdot \theta)\theta = z_y$. 
\item Using the fact that $A$ is a subset of $\partial_\theta \Omega$ we obtain the existence of $r_y>0$ such that $z_y - t\theta \in \Omega$ for any $t \in (0,r_y)$. Therefore $] (z_y\cdot \theta) - r_y, (z_y\cdot \theta)[$ is a non-empty open subset of $\omega_\theta(y)$.
\item We then have $z_\theta(s\theta+y) = z_y$ for any $s\in ] (z_y\cdot \theta) - r_y, (z_y\cdot \theta)[$.
\end{enumerate}
We thus obtain 
\begin{equation*}
   \int_{ ] (z_y\cdot \theta) - r_y, (z_y\cdot \theta)[} \chi_{A}( z_\theta(s\theta+y)) {\rm d}\lambda_1(s) = 
 \int_{ ] (z_y\cdot \theta) - r_y, (z_y\cdot \theta)[} \chi_{A}( z_y) {\rm d}\lambda_1(s) = r_y.
\end{equation*}
This implies $$
 \int_{ \omega_\theta(y)} \chi_{A}( z_\theta(s\theta+y)) {\rm d}\lambda_1(s) >0,~\text{for any}~y \in \mathcal{P}_\theta A.$$

\textbf{Step 4.} We obtain
$$
\mu_\theta(A) =  \int_{\mathcal{P}_\theta A} \Big( \int_{\omega_\theta(y)} \chi_{A}( z_\theta(s\theta+y))  {\rm d}\lambda_1(s) \Big) {\rm d}\lambda^{d-1}(y) 
$$
with 
$$
\int_{\omega_\theta(y)} \chi_{A}( z_\theta(s\theta+y))  {\rm d}\lambda_1(s) > 0,~\text{for any}~y \in \mathcal{P}_\theta A.
$$
We can now prove the lemma. If $\mu_\theta(A)>0$, the preceding equality imposes that $\lambda^{d-1}(\mathcal{P}_\theta(A))>0$. Reciprocally, if $\lambda^{d-1}(\mathcal{P}_\theta(A))>0$,
then the preceding equality imposes $\mu_\theta(A)> 0$. 
 
\end{proof}

We have the following corollary.

\begin{corollary}\label{cor:negl2} We have the following properties.  
\begin{enumerate}
\item 
The measurable set $\partial\Omega\setminus \partial_\theta\Omega$ is $\mu_\theta$-negligible (which means that $\mu_\theta(\partial\Omega\setminus \partial_\theta\Omega)=0.$)
\item 
  Let $A$ be a non-empty subset of $ \partial \Omega$. Then $A$ is $\mu_\theta$-negligible 
  if and only if $\mathcal{P}_\theta (A \cap \partial_\theta \Omega)$ is $\lambda^{d-1}$-negligible.
 \end{enumerate}
\end{corollary}

\begin{proof}
 We have $ \partial\Omega\setminus \partial_\theta\Omega=\{ z \in \partial \Omega ~\text{such that for all}~r>0,\text{there exists}~ t\in (0,r)~\text{ with}~z-t\theta\notin\Omega\}$. Therefore, by definition of $z_\theta$ the set $z_\theta(\Omega) \cap ( \partial\Omega\setminus \partial_\theta\Omega) $ is empty. This implies that $\mu_\theta(\partial\Omega\setminus\partial_\theta\Omega)=0.$

 \medskip
 
Let $A$ be a non-empty subset of $ \partial \Omega$. Let us assume that $A$ is $\mu_\theta$-negligible in $(\partial \Omega, \mathcal{B}(\partial \Omega), \mu_\theta)$. This means that there exists $N \in \mathcal{B}(\partial \Omega)$ such that $A\subset N$ and $\mu_\theta(N)=0$.   Then $A \cap \partial_\theta \Omega$ is a subset of $N \cap \partial_\theta \Omega \in {\mathcal B}(\partial_\theta\Omega)$ and $\mu_\theta(N \cap \partial_\theta \Omega)=0$. Applying Lemma \ref{lem:negl}, we obtain that $\mathcal{P}_\theta (A \cap \partial_\theta \Omega)$ is  a subset of $\mathcal{P}_\theta (N \cap \partial_\theta \Omega)$ with $\lambda^{d-1} (\mathcal{P}_\theta (N \cap \partial_\theta \Omega))=0$. This implies that $\mathcal{P}_\theta (A \cap \partial_\theta \Omega)$ is $\lambda^{d-1}$-negligible in $(H_\theta, \mathcal{B}(H_\theta),\lambda^{d-1})$.

 \medskip
 
Let us assume $\mathcal{P}_\theta (A \cap \partial_\theta \Omega)$ is $\lambda^{d-1}$-negligible in $(H_\theta, \mathcal{B}(H_\theta),\lambda^{d-1})$. There exists $N \in \mathcal{B}(H_\theta) $ such that $\mathcal{P}_\theta (A \cap \partial_\theta \Omega)$ is a subset of $N$ and $\lambda^{d-1}(N)=0$. We have $A \cap \partial_\theta \Omega$ is a subset of $\mathcal{P}_\theta^{-1} N \cap \partial_\theta \Omega \in \mathcal{B}(\partial_\theta \Omega)$ and $\mathcal{P}_\theta (\mathcal{P}_\theta^{-1} N \cap \partial_\theta \Omega)$ is a subset of $N$. Applying Lemma \ref{lem:negl}, we obtain that $\mu_\theta ( \mathcal{P}_\theta^{-1} N \cap \partial_\theta \Omega)=0$ which implies that the set $A \cap \partial_\theta \Omega $ is $\mu_\theta$-negligible. Using the fact that $A \cap \partial\Omega\setminus \partial_\theta\Omega$ is $\mu_\theta$-negligible we obtain that the set $A$ is $\mu_\theta$-negligible in $(\partial \Omega, \mathcal{B}(\partial \Omega), \mu_\theta)$.
 
\end{proof}

 \begin{example}{\bf Case of a regular domain.}\label{exa:cun}
 Let $\Omega$ be such that there exists a unit outward normal ${\bm n}(z)$ to the boundary for $\mathcal{H}^{d-1}$-a.e. $z\in\partial\Omega$ such that a Green formula holds. Let  $\varphi\in C^0(\partial\Omega)$ be given.
 We define the function ${\bm u}:\Omega\to \mathbb{R}^d$ such that ${\bm u}(x) = \big((x- z_{\minus\theta}(x))\cdot\theta\big) \varphi(z_{\theta}(x))\theta$. We notice that ${\bm u}\in H_{\rm div}(\Omega)$ with ${\rm div}{\bm u}(x) = \varphi(z_{\theta}(x))$ (it suffices to consider a basis with $\theta$ as first vector). We also notice that the function ${\bm u}$ has a normal trace $\gamma{\bm u}\in L^\infty(\partial\Omega,\mathcal{H}^{d-1})$ which is such that $\gamma{\bm u}(z) = \varphi(z) \ell_\theta(z) \max(\theta\cdot {\bm n}(z),0)$ for $\mathcal{H}^{d-1}$-a.e. $z\in\partial\Omega$. Then, applying the divergence theorem (also named the  Gauss-Green theorem \cite{pfeffer}), we get 
 \[
  \int_\Omega {\rm div}{\bm u}(x) {\rm d}x = \int_\Omega\varphi(z_\theta (x)) {\rm d}x = \int_{\partial\Omega} \varphi(z) \ell_\theta(z) \max(\theta\cdot {\bm n}(z),0) {\rm d}{\mathcal{H}}^{d-1}(z).
 \]
 This proves that $\mu_\theta(z) = \ell_\theta(z) \max(\theta\cdot {\bm n}(z),0){\mathcal{H}}^{d-1}(z)$.
 \end{example}

\begin{example}\label{ex:cantor}{\bf Complementary of Cantor set.}

For $\rho\in]0,\frac 1 3]$, we construct the $\rho$-Cantor set, denoted by $\mathcal{C}_\rho\subset [0,1]$ by the following procedure.

\begin{itemize}
 \item We define $a_1=0$ and $b_1 = 1$;
 \item We define the values $c_1 = \frac {a_1+b_1} 2 - \frac  \rho 2$ and $d_1 = \frac {a_1+b_1} 2 + \frac \rho 2$;
 \item We define $a_2 = a_1$, $b_2 = c_1$, $a_3 = d_1$, $b_3 = b_1$;
 \item For $m=2^k,\ldots,2^{k+1}-1$, we define $c_m = \frac {a_m+b_m} 2 -\frac  {\rho^{k+1}} 2$ and $d_m = \frac {a_m+b_m} 2 +\frac  {\rho^{k+1}} 2$;
 \item For $m=2^k,\ldots,2^{k+1}-1$, we define $a_{2m} = a_m$, $b_{2m} = c_m$, $a_{2m+1} = d_m$, $b_{2m+1} = b_m$;
 \end{itemize}

We set 
\[
 \Omega = \bigcup_{m\ge 1} ]c_m,d_m[.
\]
Then  $\mathcal{C}_\rho = \partial\Omega = [0,1]\setminus\Omega$, since $\Omega$ is dense in $[0,1]$. If $\rho = 1/3$, we recover the well-known ``middle third Cantor set''. 

The Lebesgue measure of $\Omega$ is equal to $\frac {\rho} {1 - 2\rho}$, therefore the Lebesgue measure of $\mathcal{C}_\rho$ is equal to $\frac {1-3\rho} {1 - 2\rho}$.

Then the unit sphere is $\mathcal{S} = \{-1,1\}$. The measure $\mu_1$ (resp. $\mu_{\minus 1}$) is the discrete measure on points $d_m$ (resp. $c_m$) with weight $d_m-c_m$. This means that, for any $A \in \mathcal{B}(\partial \Omega)$, 
$$
\mu_1(A)= \sum_{m \ge 1} (d_m-c_m) \chi_A(d_m)~\text{and}~\mu_{\minus 1}(A)=\sum_{m \ge 1} (d_m-c_m) \chi_A(c_m).
$$
The Hausdorff measure $\mathcal{H}^{0}$ (which is the counting measure) of $ \partial\Omega$ is infinite, the Hausdorff measure $\mathcal{H}^{1}$ (which coincide with the Lebesgue measure) of $ \partial\Omega$ is null for $\rho = 1/3$ and strictly positive for $\rho < 1/3$.

\end{example}
 \section{Directional trace}\label{sec:traceund}
 
 We now choose in this section a direction $\theta\in \mathcal{S}$. 
 
 The strong sense of the directional derivative along $\theta$ of any function $\varphi\in C^\infty_c(\Omega)$, denoted in this section by $\partial_\theta   \varphi$, is defined by
 \[
  \forall x\in\Omega,\ \partial_\theta   \varphi(x) = \lim_{h\to 0} \frac 1 h( \varphi(x+h\theta)-\varphi(x)) = \theta\cdot\nabla\varphi(x),
 \]
 using the Euclidean scalar product in $\mathbb{R}^d$.
 For any $u\in L^2(\Omega)$, we say that $u$ admits a directional derivative along $\theta$  in the weak sense if there exists $v\in L^2(\Omega)$ such that
 \begin{equation}\label{eq:defadj}
  \forall \varphi\in C^\infty_c(\Omega),\ \int_\Omega u(x) \partial_\theta   \varphi(x) {\rm d}x =  -\int_\Omega v(x) \varphi(x) {\rm d}x.
 \end{equation}
 In this case, $v$ is uniquely defined, and we denote by $\partial_\theta   u = v$. Note that, in the case $u\in  C^\infty_c(\Omega)$, we retrieve the same function as the one defined by the strong sense. We then define 
 \[
W_\theta(\Omega) = \{ u \in L^2(\Omega); \partial_\theta   u\in  L^2(\Omega)\}.
\]
 Then $W_\theta(\Omega)$ is a Hilbert space, with the scalar product
 \[
  \langle u,v\rangle_\theta = \int_\Omega (u(x) v(x) + \partial_\theta   u(x)\partial_\theta   v(x)){\rm d}x.
 \]
 Let us observe that $W_\theta(\Omega) = W_{\minus\theta}(\Omega)$, $ \partial_{\minus\theta} = -\partial_\theta $, $H_\theta =H_{\minus\theta} $, $\mathcal{P}_\theta =\mathcal{P}_{\minus\theta} $  and $\langle u,v\rangle_\theta =\langle u,v\rangle_{\minus\theta}$. Finally, we observe that, if $z\in \partial_\theta\Omega$, then $z_{\minus\theta}(z)\in \partial_{\minus\theta}\Omega$ and that, if $]\alpha,\beta[\in \mathcal{I}_{\theta}(y)$, then $]-\beta,-\alpha[\in \mathcal{I}_{\minus\theta}(y)$.
 
 Notice that $\partial_\theta\Omega,\mu_\theta$ on one hand, and $\partial_{\minus\theta}\Omega,\mu_{\minus\theta}$ on the other hand, are different in the general case. Indeed, it may hold $\partial_\theta\Omega\cap\partial_{\minus\theta}\Omega = \emptyset$ in some situations. For example, if $\Omega$ is $C^1$ regular, $\partial_\theta\Omega = \{z\in\partial\Omega;\theta\cdot{\bm n}(z)>0\}$ and $\partial_{\minus\theta}\Omega = \{z\in\partial\Omega;\theta\cdot{\bm n}(z)<0\}$. But there are situations where $\partial_\theta\Omega\cap\partial_{\minus\theta}\Omega \neq \emptyset$ (see Examples \ref{exa:fisund} and  \ref{exa:fisdeuxd}). A case where $\partial_\theta\Omega = \partial_{\minus\theta}\Omega$ is given by $\theta = \pm 1$ and $\Omega = \bigcup_{n\in\mathbb{N}^\star}]-1+\frac 1 {n+1}, -1+\frac 1 n[\cup\bigcup_{n\in\mathbb{N}^\star}]1-\frac 1 n, 1-\frac 1 {n+1}[$.

\begin{definition}[Directional trace in the direction $\theta$]\label{def:dirtrace}
 
 Let $u\in W_\theta(\Omega)$. We say that a function $g:\partial\Omega\to\mathbb{R}$ is a directional trace of $u$ in the direction $\theta$ if, there exists a representative of $u$, again denoted by $u$, and there exists a set $A\subset\mathcal{P}_\theta(\Omega)$ such that (see notation of section \ref{sec:openbounded}):
 \begin{align}
  \mbox{the set $\mathcal{P}_\theta(\Omega)\setminus A$ is $\lambda^{d-1}$-negligible}\label{eq:dirtraceneg},\\
  \mbox{ for any $y\in A$ and any $]\alpha,\beta[\in \mathcal{I}_{\theta}(y)$, $f:]\alpha,\beta[\to\mathbb{R}$, $s\mapsto u(s\theta+y)$ is such that $f\in H^1(]\alpha,\beta[)$,}\nonumber\\
  \mbox{and $g$ is such that $g(\beta\theta+y) = T_\beta(f)$}\label{eq:dirtraceeq},
  \end{align}
  where, for any $\alpha<\beta\in\mathbb{R}$ and $f\in H^1(]\alpha,\beta[)$, we denote by $T_\beta(f)$ the value in $\beta$ of the representative of $f$ which belongs to $C^0([\alpha,\beta])$. 
We then define the directional trace of $u$ in the direction $\theta$, denoted by $\gamma_\theta u$, as the set of all directional traces of $u$ in the direction $\theta$ in the preceding sense.
\end{definition}

Note that the same definition applies, changing $\theta$ in $-\theta$.
In Definition \ref{def:dirtrace}, we can in fact choose any representative of $u$ as stated in the following lemma.

\begin{lemma}\label{lem:traceoned} Let $u\in  W_\theta(\Omega)$ and let $g:\partial\Omega\to\mathbb{R}$ be a directional trace in the direction $\theta$ in the sense of Definition \ref{def:dirtrace}. Let $u_1$ be a representative of $u$ and let $A_1\subset\mathcal{P}_\theta(\Omega)$ be such that \eqref{eq:dirtraceneg}-\eqref{eq:dirtraceeq} hold with $u=u_1$ and $A=A_1$. Then, for any representative $u_2$ of $u$, there exists $A_2\subset\mathcal{P}_\theta(\Omega)$ such that \eqref{eq:dirtraceneg}-\eqref{eq:dirtraceeq} hold with $u=u_2$ and $A = A_2$.
\end{lemma}
\begin{proof}
 Let $\widehat{A}\subset\mathcal{P}_\theta(\Omega)$, whose the complementary in  $\mathcal{P}_\theta(\Omega)$ is  $\lambda^{d-1}$-negligible, be such that, for all $y\in \widehat{A}$, $u_1(s\theta+y) = u_2(s\theta+y)$ for a.e. $s\in \omega_\theta(y)$. Then for all  $y\in A_2:=\widehat{A}\cap A_1$ (whose complementary in  $\mathcal{P}_\theta(\Omega)$ is therefore $\lambda^{d-1}$-negligible), for any $]\alpha,\beta[\in \mathcal{I}_{\theta}(y)$, the function $f:]\alpha,\beta[\to\mathbb{R}$, $s\mapsto u_1(s\theta+y)$ is such that $f(s) = u_2(s\theta+y)$ for a.e. $s\in ]\alpha,\beta[$, $f\in H^1(]\alpha,\beta[)$ and $g(\beta\theta+y) = T_\beta(f)$. Hence  \eqref{eq:dirtraceneg}-\eqref{eq:dirtraceeq} hold with $u = u_2$ and $A = A_2$.
\end{proof}

 In Example \ref{exa:cuspidal}, we give an example of directional trace for a function which belongs to $W_\theta(\Omega)$ in the case of a non-Lipschitz domain. 
 
 \medskip
 
 The following lemma provides a link between the notions of directional trace and directional measure introduced in Section \ref{sec:openbounded}.

\begin{theorem}[Directional traces and measures $\mu_\theta$ and $\mu_{\minus\theta}$]\label{thm:deftraceoned}
 Let $u\in  W_\theta(\Omega) = W_{\minus\theta}(\Omega)$. Then there exists a directional trace $g$ (resp. $g_{\minus}$) $:\partial\Omega\to\mathbb{R}$ of $u$ in the direction $\theta$ (resp. $-\theta$) in the sense of Definition \ref{def:dirtrace}, which moreover is measurable for the Borel $\sigma$-algebra. 
 
 Moreover, any function $\widetilde{g}:\partial\Omega\to\mathbb{R}$ (resp. $\widetilde{g}_{\minus}$) is a directional trace of  $u$ in the direction $\theta$ (resp. $-\theta$)  in the sense of Definition \ref{def:dirtrace} if and only if $\widetilde{g}(z) = g(z)$ for $\mu_\theta$-a.e. $z\in\partial\Omega$ (resp. $\widetilde{g}_{\minus}(z) = g_{\minus}(z)$ for $\mu_{\minus\theta}$-a.e. $z\in\partial\Omega$).
 
 Therefore the directional trace of $u$ in the direction $\theta$ (resp. $-\theta$) denoted by  $\gamma_\theta u$ (resp. $\gamma_{\minus\theta} u$) in Definition \ref{def:dirtrace} is an equivalence class, for the relation $\mu_\theta$-a.e. equal (resp. $\mu_{\minus\theta}$-a.e. equal),  of a measurable function for the Borel $\sigma$-algebra.
\end{theorem}
\begin{proof} We only write the proof for $\theta$, the proof for $\minus\theta$ is obtained by changing $\theta$ in $-\theta$ everywhere.

\medskip

{\bf Step 1: existence of a directional trace.}

\medskip

We again denote by $u$ a representative of $u$.
Let $\psi$ be any isometry from $\mathbb{R}^{d-1} \to H_\theta\subset\mathbb{R}^d$ (recall that $H_\theta$ is the hyperplane which is orthogonal to $\theta$ and contains the point $0$), let 
\[
 \widetilde{\Omega} = \{ (s,\widetilde{y})\in\mathbb{R}\times \mathbb{R}^{d-1},\ s\theta+ \psi(\widetilde{y})\in\Omega\}.
\]
Let $v$ be the function defined on $\widetilde{\Omega}$ by  $v(s,\widetilde{y}) = u(s\theta + \psi(\widetilde{y}))$. 
Then, using the notations of Appendix \ref{ap:hund}, we get that $v$ is a representative of an element of $W_1(\widetilde{\Omega})$ with $\partial_1 v(s,\widetilde{y}) = \partial_\theta u(s\theta + \psi(\widetilde{y}))$. 
Applying Lemma \ref{lem:underivative} to $v$, we get the existence of the measurable set $A_1\subset \mathcal{P}_1(\widetilde{\Omega})\subset\mathbb{R}^{d-1}$ such that the complementary of $A_1$ in $\mathcal{P}_1(\widetilde{\Omega})$ is  $\lambda^{d-1}$-negligible, and for all $\widetilde{y}\in A_1$, $v(\cdot,\widetilde{y})\in H^1(\omega_1(\widetilde{y}))$. Denoting by 
$A = \psi(A_1)$, we get that, for any $y\in A$, it holds $u(\cdot~\theta+y)\in  H^1(\omega_\theta(y))$, since $\omega_\theta(y) = \omega_1(\psi^{-1}(y))$. Notice that the complementary of $A$ in $\mathcal{P}_\theta(\Omega)$ is  $\lambda^{d-1}$-negligible. We then define $g:\partial\Omega\to\mathbb{R}$ by setting, for any $y\in A$ and any $]\alpha,\beta[\in \mathcal{I}_{\theta}(y)$, $g(\beta\theta+y) = T_\beta(f)$, where  $f(s) = u(s\theta+y)$ for a.e. $s\in ]\alpha,\beta[$. We then prescribe $g(z) = 0$ for any other $z\in\partial\Omega$. 

\medskip

We then observe that, for any $s\in ]\alpha,\beta[$ and $x = s\theta + y$, we have $z_{\minus\theta}(x) = \alpha\theta+y$, $z_{\theta}(x) = \beta\theta+y$ and $\beta-\alpha = \delta_{\minus\theta}(x)+\delta_{\theta}(x)$. Therefore we get from \eqref{eq:valbeta} that, for all $x\in\Omega$ such that $\mathcal{P}_\theta(x)\in A$, the following holds:
\begin{equation}\label{eq:traceoned}
  g(z_\theta(x)) = \frac 1 {\delta_{\minus\theta}(x)+\delta_{\theta}(x)}\int_{\minus\delta_{\minus\theta}(x)}^{\delta_{\theta}(x)} \big(u(t~\theta+x) + (t+\delta_{\minus\theta}(x))\partial_\theta u(t~\theta+x)\big){\rm d}t,
 \end{equation}
 which leads to the measurability of $g$ for the Borel $\sigma$-algebra.

\medskip

{\bf Step 2: all directional traces are $\mu_\theta$-a.e. equal. }

\medskip

Let $g$ (resp. $\widetilde{g}$) be two directional traces of  $u$ in the direction $\theta$  in the sense of Definition \ref{def:dirtrace}. Let $u$ again denote a representative of $u$, and let $A$ (resp. $\widetilde{A}$) be a subset of $\mathcal{P}_\theta(\Omega)$ such that \eqref{eq:dirtraceneg}-\eqref{eq:dirtraceeq} hold for $g$, $u$ and $A$ on one hand, $\widetilde{g}$, $u$ and $\widetilde{A}$ on the other hand (owing to Lemma \ref{lem:traceoned}, we can consider the same representative of $u$ for satisfying \eqref{eq:dirtraceneg}-\eqref{eq:dirtraceeq} with $g$ and $\widetilde{g}$). Then the complementary of $A\cap \widetilde{A}$ in $\mathcal{P}_\theta(\Omega)$ is $\lambda^{d-1}$-negligible, and, for all $y\in A\cap \widetilde{A}$ and a.e. $s\in \omega_\theta(y)$ and for any $]\alpha,\beta[\in \mathcal{I}_{\theta}(y)$, then $g(\beta\theta+y) = \widetilde{g}(\beta\theta+y)$.  Using Lemma \ref{lem:partialthetaboundary}, we obtain that the set $\{ z \in \partial_\theta \Omega~\text{such that}~g(z)=\tilde g(z)~\text{and}~\mathcal{P}_\theta z \in A \}$ is a subset of $\{ g= \tilde g\} \cap \partial_\theta \Omega$. Applying corollary \ref{cor:negl2}, this implies that $g(z) = \widetilde{g}(z)$ for $\mu_\theta$-a.e. $z\in\partial\Omega$.

\medskip

{\bf Step 3: any function $\mu_\theta$-a.e. equal  to a directional trace is a directional trace.}

\medskip
Let $g$ be a directional trace of  $u$ in the direction $\theta$  in the sense of Definition \ref{def:dirtrace}. Let  $\widetilde{g}:\partial\Omega\to\mathbb{R}$ be such that $g(z) = \widetilde{g}(z)$ for $\mu_\theta$-a.e. $z\in\partial\Omega$. We then define
\[
 \widehat{A}= \{y\in \mathcal{P}_\theta(\Omega), \forall ]\alpha,\beta[\in\mathcal{I}_\theta(y), g(\beta\theta + y) = \widetilde{g}(\beta\theta + y) \}.
\] 
From  Lemma \ref{lem:negl}, we get that the complementary of $\widehat{A}$ in $\mathcal{P}_\theta(\Omega)$ is $\lambda^{d-1}$-negligible. 
Considering a representative of $u$ again denoted by $u$, let $A$  be a subset of $\mathcal{P}_\theta(\Omega)$ such that \eqref{eq:dirtraceneg}-\eqref{eq:dirtraceeq} are satisfied for $g$, $u$ and $A$. We then have the following:
\begin{itemize}
  \item the complementary of $A\cap  \widehat{A}$ in $\mathcal{P}_\theta(\Omega)$ is $\lambda^{d-1}$-negligible,
  \item for any $y\in A\cap \widehat{A}$, we have $u_{\theta,y}\in H^1(\omega_\theta(y))$; then, for any $]\alpha,\beta[\in \mathcal{I}_{\theta}(y)$, since $\beta\theta+y \in \partial_\theta\Omega$ and $y\in \widehat{A}$, we have $ \widetilde{g}(\beta\theta+y) = g(\beta\theta+y) =  T_\beta(f)$, where $f(s) = u(s\theta+y)$ for a.e. $s\in ]\alpha,\beta[$.
 \end{itemize}

 Therefore \eqref{eq:dirtraceneg}-\eqref{eq:dirtraceeq} hold with $\widetilde{g}$, $u$ and $A\cap \widehat{A}$, which shows that $\widetilde{g}$ is a directional trace of  $u$ in the direction $\theta$  in the sense of Definition \ref{def:dirtrace}.
\end{proof}

\begin{figure}[!ht]
\begin{center}
\resizebox{.5\linewidth}{!}{
 \includegraphics{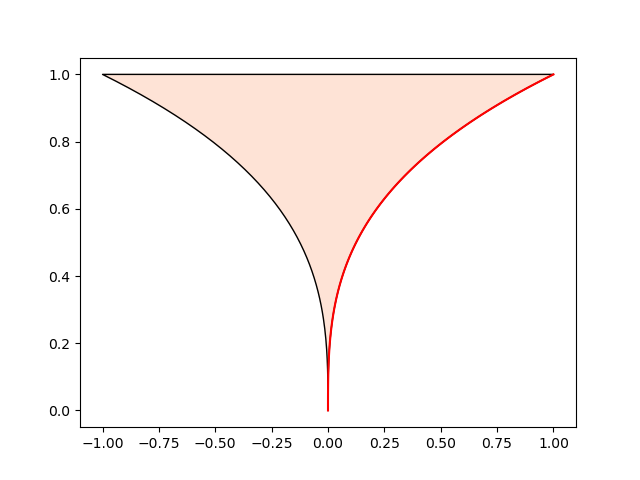}
 }
 \end{center}
 \caption{Cuspidal domain $\Omega$ in Example \ref{exa:cuspidal}. \label{fig:triangle}}
\end{figure}
 \begin{example}[Case of a cuspidal domain]\label{exa:cuspidal}
 Let us assume that $\Omega\subset\mathbb{R}^2$ is the (non-Lipschitz) domain defined by (see Figure \ref{fig:triangle})
 \[
  \Omega = \{(x_1,x_2)\in \mathbb{R}\times]0,1[; -x_2^3 < x_1 < x_2^3\}.
 \]
 Assume that $\theta = (1,0)$. Then we get for any $(x_1,x_2)\in\Omega$ that $z_{\theta}(x_1,x_2) = (x_2^3,x_2)$, $\ell_{\theta}(x_1,x_2) = 2 x_2^3$. Considering the right side of the domain (in red in  Figure \ref{fig:triangle}), we have $\theta\cdot {\bm n}(x_2^3,x_2) = \frac 1 {\sqrt{1 + 9 x_2^4}}$, the measure $\mu_{\theta}(x_2^3,x_2)$ is equal to $\frac {2 x_2^3} {\sqrt{1 + 9 x_2^4}}  \mathcal{H}^1(x_2^3,x_2)$ for a.e. $x_2\in ]0,1[$, and $\mu_{\theta}$ is equal to $0$ on the remaining of the boundary (in black in  Figure \ref{fig:triangle}). Then the function $u(x_1,x_2) = x_2^{-\alpha}$ belongs to $H^1(\Omega)\subset W_\theta(\Omega)$ for any $\alpha\in ]\frac 1 2,1[$ (indeed, the function $x_2^{-2\alpha - 2} \times x_2^3$ is integrable on $]0,1[$), and a directional trace of $u$ in the sense of Definition \ref{def:dirtrace} is the function $g(x_1,x_2) = x_2^{-\alpha}$ on all edges. In this case, $g$ is also equal a.e. to the trace of $u$ as defined in Section \ref{sec:tracehun}.  Note that
 $g\in \mathcal{L}^2(\partial\Omega,\mu_\theta)$ but $g\notin \mathcal{L}^2(\partial\Omega,\mathcal{H}^1)$ (the function $x_2^{-2\alpha}$ is not integrable on $]0,1[$).
 \end{example}

\begin{theorem}\label{thm:proptraceoned}
 The operators  $\gamma_\theta$ and $\gamma_{\minus\theta}$ defined in Theorem \ref{thm:deftraceoned}, for all $u\in W_\theta(\Omega)$, as the directional traces in the respective directions $\theta$ and $-\theta$ on the boundary, satisfy the following properties.
 \begin{itemize}
  \item  For any $u\in W_\theta(\Omega)$, then $\gamma_{\theta} u\in L^2(\partial\Omega,\mu_{\theta} )$ and 
   \begin{equation}\label{eq:ineqtracezero}
\int_{\partial\Omega} (\gamma_{\theta} u(z))^2 {\rm d}\mu_\theta(z)\le 2 \max(1,{\rm diam}(\Omega)^2)\Vert u\Vert_\theta^2,
\end{equation}
\item For any $u\in  W_\theta(\Omega)$ such that $\gamma_{\theta} u(z) = 0$ for $\mu_\theta$-a.e. $z\in\partial\Omega$, then the following Poincar\'e inequality holds:
\begin{equation}\label{eq:pointheta}
 \Vert u\Vert_{L^2(\Omega)} \le {\rm diam}(\Omega) \Vert \partial_\theta  u\Vert_{L^2(\Omega)}.
 \end{equation}
\item  For any $u\in W_\theta(\Omega)$, the pair of functions  $(\gamma_{\theta} u, \gamma_{\minus \theta} u)$ satisfies the following inequalities:
\begin{equation}\label{eq:ineqtraceun}
 \int_{\partial\Omega} \Big(\gamma_{\theta} u(z) +\gamma_{\minus \theta} u(z_{\minus \theta}(z))\Big)^2 {\rm d}\mu_\theta(z)\le 4 \max(1,{\rm diam}(\Omega)^2)\Vert u\Vert_\theta^2,
\end{equation}
and
\begin{equation}\label{eq:ineqtracedeux}
 \int_{\partial\Omega} \Big(\frac{\gamma_{\theta} u(z) -\gamma_{\minus \theta} u(z_{\minus \theta}(z))}{|z -z_{\minus \theta}(z)|}\Big)^2 {\rm d}\mu_\theta(z)\le \Vert u\Vert_\theta^2.
\end{equation}
\item For any  $u,v\in  W_\theta(\Omega)$, we have
 \begin{equation}\label{eq:greenoned}
 \int_\Omega (u(x)\partial_\theta v(x) +v(x)\partial_\theta u(x)){\rm d}x = \int_{\partial\Omega} \frac{\gamma_{\theta} u(z)\gamma_{\theta} v(z) -\gamma_{\minus \theta} u(z_{\minus \theta}(z))\gamma_{\minus \theta} v(z_{\minus \theta}(z))}{|z -z_{\minus \theta}(z)|} {\rm d}\mu_\theta(z).
 \end{equation}
 
 \end{itemize}
 The same properties hold for $-\theta$, changing $\theta$ in $-\theta$.
\end{theorem}
\begin{proof} 
Let  $u\in W_\theta(\Omega)$, let $g$ be a representative of $\gamma_{\theta} u$ which is a directional trace of $u$ in the sense of Definition \ref{def:dirtrace}.
We again denote by $u$ a representative of $u$ and by $A\subset \mathcal{P}_\theta(\Omega)$ a set such that \eqref{eq:dirtraceneg}-\eqref{eq:dirtraceeq} hold. We have, using the notation of  Definition \ref{def:dirtrace} and accounting for \eqref{eq:defmutheta},
\begin{multline*}
 \int_{\partial\Omega} g(z)^2 {\rm d}\mu_\theta(z) = \int_\Omega g( z_\theta(x))^2 {\rm d}x  = \int_{H_\theta} \Big( \sum_{]\alpha,\beta[\in  \mathcal{I}_{\theta}(y)} \int_{]\alpha,\beta[} g(\beta\theta+y)^2 {\rm d}\lambda_1(s) \Big) {\rm d}\lambda^{d-1}(y)
 \\ = \int_{H_\theta} \Big( \sum_{]\alpha,\beta[\in  \mathcal{I}_{\theta}(y)} (\beta - \alpha)~ g(\beta\theta+y)^2 \Big) {\rm d}\lambda^{d-1}(y) .
\end{multline*}
Applying \eqref{eq:majfhun} in Lemma \ref{lem:fhun} to the function $ u_{\theta,y}$, we get
\begin{multline*}
 \int_\Omega g( z_\theta(x))^2 {\rm d}x  \le  \int_{H_\theta} \Big( \sum_{]\alpha,\beta[\in  \mathcal{I}_{\theta}(y)} \Vert u_{\theta,y}\Vert_{H^1(]\alpha,\beta[)}^2 2\max(1, (\beta - \alpha)^2){\rm d}\lambda^{d-1}(y) 
 \\
 \le 2 \max(1,{\rm diam}(\Omega)^2)  \int_{H_\theta}  \Vert u_{\theta,y}\Vert_{H^1(\omega_\theta(y))}^2 {\rm d}\lambda^{d-1}(y).
\end{multline*}
We then notice that
\[
 \Vert u_{\theta,y}\Vert_{H^1(\omega_\theta(y))}^2 = \int_{\omega_\theta(y)} (u(s\theta+y)^2+\partial_\theta u(s\theta+y)^2){\rm d}\lambda_1(s).
\]
This proves \eqref{eq:ineqtracezero}. We then observe that, for any $u_{\theta,y}\in H^1(\omega_\theta(y))$ such that $g(\beta\theta+y)=0$,  \eqref{eq:poinfhun} implies
\[
 \Vert u_{\theta,y}\Vert_{L^2(\omega_\theta(y))}^2\le {\rm diam}(\Omega)^2 \Vert Du_{\theta,y}\Vert_{L^2(\omega_\theta(y))}^2.
\]
Since $\beta - \alpha\le {\rm diam}(\Omega)$ and integrating with respect to $y\in\mathcal{P}_\theta(\Omega)$ provides \eqref{eq:pointheta}.

The proofs of \eqref{eq:ineqtraceun} and \eqref{eq:ineqtracedeux} are similar, accounting for \eqref{eq:majsumfhun} and \eqref{eq:majdiffhun} in Lemma \ref{lem:fhun}. Finally, we observe that the following integration by parts relation holds for any $f,g\in H^1(]\alpha,\beta[)\cap C^0([\alpha,\beta])$,
\[
 \int_{]\alpha,\beta[} (f(s) Dg(s) + Df(s) g(s)){\rm d}\lambda_1(s) = f(\beta)g(\beta) -f(\alpha)g(\alpha).
\]
Noticing that, letting $\ell = \beta-\alpha$, $a = f(\beta)$, $b = f(\alpha)$, $a' =  g(\beta)$, $b' = g(\alpha)$, we have
  \[
  \big| \frac {a a' - b b'} {\ell}\big| = \frac 1 2 \Big|\frac {a-b} {\ell} (a'+b') + \frac {a'-b'} {\ell} (a+b)\Big|\le \frac 1 4 \Big((\frac {a-b} {\ell})^2 + (a'+b')^2 + (\frac {a'-b'} {\ell})^2+ (a+b)^2\Big),
  \]
  we get from inequalities \eqref{eq:ineqtraceun} and \eqref{eq:ineqtracedeux} that the right-hand-side of \eqref{eq:greenoned} is well defined. Hence, we can obtain \eqref{eq:greenoned}, following the same integration rules as those used in the proof of  \eqref{eq:ineqtracezero}.
 \end{proof}

 \begin{remark} As noticed in the introduction, we cannot give a sense in \eqref{eq:greenoned} to $\int_{\partial\Omega} \frac{\gamma_{\theta} u\gamma_{\theta}v}{\ell_\theta}  {\rm d}\mu_\theta$ nor to $\int_{\partial\Omega} \frac{\gamma_{\minus\theta} u\gamma_{\minus\theta} v}{\ell_{\minus\theta}}  {\rm d}\mu_{\minus\theta}$ since we need to combine both terms for getting integrable functions. Consider the domain in Example \ref{exa:cuspidal}, and the functions $u(x_1,x_2) = v(x_1,x_2) = x_2^{-\alpha}$ with $\alpha\in ]\frac 1 2,1[$, then $\frac{\gamma_{\theta} u\gamma_{\theta} v}{\ell_\theta} $ is not integrable for the measure $\mu_\theta$.
 \end{remark}
 \begin{remark}
 Note that $W_\theta(\Omega)$ is the domain of the operator $\partial_\theta$ defined by \eqref{eq:defadj} as the adjoint in $L^2(\Omega)$ of the skew-symmetric operator $-\partial_\theta$ with domain $C^\infty_c(\Omega)$. Hence we know from \cite{ace2023ext} that there exist Hilbert spaces $H_\plus$ and $H_\minus$ and two continuous operators $G_\pm:W_\theta(\Omega)\to H_\pm$ such that the mapping $u\mapsto (G_\plus(u),G_\minus(u))$ defined from $W_\theta(\Omega)$ to $H_\plus\times H_\minus$ is surjective, and
 \[
  \forall u,v\in W_\theta(\Omega),\ \int_\Omega \big( u(x) \partial_\theta   v(x) + v(x) \partial_\theta   u(x)\big){\rm d}x =  \langle G_\plus u,G_\plus v\rangle_{H_\plus} - \langle G_\minus u,G_\minus v\rangle_{H_\minus}.
 \]
We get from Theorem \ref{thm:proptraceoned} that we can define $H_\pm = G_\pm(W_\theta(\Omega)) = L^2(\partial\Omega,\mu_\theta)$ with $G_\pm:W_\theta(\Omega)\to L^2(\partial_\theta\Omega,\mu_\theta)$ given by
\[
 G_\plus u(z) = \frac 1 4 \big(a+ b + \frac {a-b} {\ell}\big)\hbox{ and }G_\minus u(z) = \frac 1 4 \big(a+ b - \frac {a-b} {\ell}\big),
\]
for $\mu_\theta$-a.e. $z\in\partial\Omega$, letting $\ell = \ell_\theta(z)$, $a =  \gamma_{\theta} u(z)$ and $b = \gamma_{\minus \theta} u(z_{\minus \theta}(z))$.
The surjectivity of the mapping $u\mapsto (G_\plus(u),G_\minus(u))$, from $W_\theta(\Omega)$ to $L^2(\partial\Omega,\mu_\theta)^2$, is a consequence of the fact that, for any $a,b\in\mathbb{R}$, the function $f:s\mapsto a + (b-a)\frac {s - \alpha}{\beta - \alpha}$ is such that $f\in H^1(]\alpha,\beta[)$ with $T_\beta(f) = b$ and $T_\alpha(f) = a$ .
\end{remark}

 \begin{example}[Case of a regular domain]\label{exa:lip}
 Let us assume that $\Omega$ is a regular  domain in the sense of Example \ref{exa:cun}, again denoting by ${\bm n}(z)$ the unit outward vector for a.e. $z\in\partial\Omega$. Then the measure on $\partial\Omega$ defined by $\varphi\mapsto \int_\Omega \frac{\varphi(z_\theta (x))}{\ell_\theta(x)} {\rm d}x$ is equal to $\max(\theta\cdot {\bm n}(z),0) {\mathcal{H}}^{d-1}(z)$.   
 \end{example}

\begin{lemma}[Restriction to an open subset of $\Omega$]\label{lem:included}
 Let $\widehat{\Omega}\subset\Omega$ be an open set. We denote by  $\widehat{\mu}_\theta$,  $\widehat{z}_\theta$, $\widehat{\gamma}_\theta$ the quantities  $\mu_\theta$,  $z_\theta$, $\gamma_\theta$ replacing $\Omega$ by  $\widehat{\Omega}$.  Let $u\in H^1(\Omega)$, and let $\widehat{u}$ denote the restriction of $u$ to $\widehat{\Omega}$.  Let $g$ (resp $\widehat{g}$) be a directional trace of $u$ (resp. $\widehat{u}$) in the direction $\theta$ in the sense of Definition  \ref{def:dirtrace}. Then $\widehat{g}(z) = g (z)$ for $\mu_\theta$-a.e. and $\widehat{\mu}_\theta$-a.e. $z\in \partial_\theta\widehat{\Omega}\cap \partial_\theta\Omega$.
\end{lemma}
\begin{proof} We select a representative of $u$, again denoted by $u$, and we again denote by $\widehat{u}$ the restriction of the representative $u$ to  $\widehat{\Omega}$. Let $A$ (resp $\widehat{A}$) such that \eqref{eq:dirtraceneg}-\eqref{eq:dirtraceeq} hold, choosing the representative $u$ (resp. $\widehat{u}$).

Then for all $y\in A$, we have $u(\cdot\theta+y)\in H^1(\omega_\theta(y))$ and for all $y\in \widehat{A}$, we also have $\widehat{u}(\cdot\theta+y) = u(\cdot\theta+y)\in H^1(\widehat{\omega}_\theta(y))$, and the sets $\mathcal{P}_\theta(\partial_\theta\widehat{\Omega})\setminus \widehat{A}$ and $\mathcal{P}_\theta(\partial_\theta\Omega)\setminus A$ are $\lambda^{d-1}$-negligible.
Let us define
\[
 B = \{z\in \partial_\theta\widehat{\Omega}\cap \partial_\theta\Omega, g(z) = \widehat{g}(z)\},
\]
and
\[
 \widetilde{B} = \{z\in \partial_\theta\widehat{\Omega}\cap \partial_\theta\Omega, \mathcal{P}_\theta(z) \in A\cap \widehat{A}\}.
\]
We get from \eqref{eq:dirtraceeq} that, for $z\in \widetilde{B}$, it holds $g(z) = \widehat{g}(z)$. This proves that $\widetilde{B}\subset B$. 

Let $y\in  \mathcal{P}_\theta(\partial_\theta\widehat{\Omega}\cap \partial_\theta\Omega\setminus B)$. We therefore have that, for $z\in \partial_\theta\widehat{\Omega}\cap \partial_\theta\Omega\setminus B$ such that $y = \mathcal{P}_\theta(z)$, $g(z) \neq \widehat{g}(z)$, which means that $y\notin 
A\cap \widehat{A}$. If $y\in A$, this implies that $y\notin \widehat{A}$, which means that $y\in \partial_\theta(\widehat{\Omega})\setminus\widehat{A}$. If  $y\in \widehat{A}$, this implies that $y\notin A$, which means that $y\in \partial_\theta(\Omega)\setminus\widehat{A}$. If $y\notin A\cup\widehat{A}$, it means that
$y\in (\partial_\theta(\widehat{\Omega})\setminus\widehat{A}) \cup  (\partial_\theta(\Omega)\setminus A)$. Therefore
\[
  \mathcal{P}_\theta(\partial_\theta\widehat{\Omega}\cap \partial_\theta\Omega\setminus B) \subset (\mathcal{P}_\theta(\partial_\theta\Omega)\setminus A)\cup (\mathcal{P}_\theta(\partial_\theta\widehat{\Omega})\setminus\widehat{A}).
\]
This implies that the set $\mathcal{P}_\theta(\partial_\theta\widehat{\Omega}\cap \partial_\theta\Omega\setminus B$) is $\lambda^{d-1}$-negligible.
Therefore, from Corollary \ref{cor:negl2}, we get that
$\partial_\theta\widehat{\Omega}\cap \partial_\theta\Omega\setminus B$ is both $\mu_\theta$-negligible and $\widehat{\mu}_\theta$-negligible.
\end{proof}
\begin{remark}
 Lemma \ref{lem:included} cannot state that $\widehat{g}(z) = g(z)$ for $\mu_\theta$-a.e. and $\widehat{\mu}_\theta$-a.e. $z\in \partial\widehat{\Omega}\cap \partial\Omega$ instead of $z\in \partial_\theta\widehat{\Omega}\cap \partial_\theta\Omega$. Indeed, let us consider the case $\Omega = ]0,1[$ and $\widehat{\Omega} = \bigcup_{n\in\mathbb{N}^\star}]\frac 1 {2n+1},\frac 1{2n}[$. Then, for $\theta = -1$, $\partial_\theta\widehat{\Omega} = \{ \frac 1 {2n+1}, n\in\mathbb{N}^\star\}$,  $0\in\partial\widehat{\Omega}$ and  $\partial_\theta\Omega = \{0\}$. Then one can take any value for $\widehat{g}(0)$ such that $\widehat{g}(0)\neq g(0)$ (since $\widehat{\mu}_\theta(\{0\}) = 0$) whereas $\mu_\theta(\{0\}) = 1$.
\end{remark}

\section{Trace operator on $H^1(\Omega)$}\label{sec:tracehun}

Let us observe that
\[
 H^1(\Omega) = \bigcap_{\theta\in{\mathcal S}} W_\theta(\Omega),
\]
and that, for all $\theta\in{\mathcal S}$ and $u\in H^1(\Omega)$, we have $\Vert u\Vert_\theta \le \Vert u\Vert_{H^1(\Omega)}$. 
This leads to the following definition.

\begin{definition} Let ${\mathcal{L}}^2(\partial\Omega,(\mu_\theta)_{\theta\in\mathcal{S}})$ denote the set of all functions $f:\partial\Omega\to\mathbb{R}$ such that there exists $M\ge 0$ and for all $\theta\in \mathcal{S}$, there exists $g_\theta\in \mathcal{L}^2(\partial\Omega,\mu_\theta)$ with $f(z) = g_\theta(z)$ for $\mu_\theta$-a.e. $z\in \partial\Omega$ and $\int_\Omega |g_\theta|^2 {\rm d}\mu_\theta\le M$. 

We say that, for $f,g \in {\mathcal{L}}^2(\partial\Omega,(\mu_\theta)_{\theta\in\mathcal{S}})$, $f$ is equivalent to $g$ if, for all $\theta\in \mathcal{S}$, $f(z) = g(z)$ for $\mu_\theta$-a.e. $z\in \partial\Omega$ and we define $L^2(\partial\Omega,(\mu_\theta)_{\theta\in\mathcal{S}})$ as the set of all equivalence classes of the elements of ${\mathcal{L}}^2(\partial\Omega,(\mu_\theta)_{\theta\in\mathcal{S}})$, normed by
\[
 \forall g\in L^2(\partial\Omega,(\mu_\theta)_{\theta\in\mathcal{S}}),\ \Vert g\Vert_{L^2(\partial\Omega,(\mu_\theta)_{\theta\in\mathcal{S}})}^2 = \sup_{\theta\in\mathcal{S}}\int_\Omega |g|^2 {\rm d}\mu_\theta
\]

\end{definition}
Note that, for $f,g \in {\mathcal{L}}^2(\partial\Omega,(\mu_\theta)_{\theta\in\mathcal{S}})$,  $f$ is equivalent to $g$ if, for all $\theta\in \mathcal{S}$, $f(z) = g(z)$ for $\mu_\theta$-a.e. $z\in \widetilde{\partial\Omega}$ where $\widetilde{\partial\Omega}\subset \partial\Omega$ is defined by \eqref{eq:defdomegatilde}, which is a dense subset of $\partial\Omega$ (see Lemma \ref{lem:closomegatilde}).

The set $L^2(\partial\Omega,(\mu_\theta)_{\theta\in\mathcal{S}})$ satisfies the following property.
\begin{lemma}\label{lem:ldsbanach}
The space $L^2(\partial\Omega,(\mu_\theta)_{\theta\in\mathcal{S}})$ is a Banach space.
 
\end{lemma}
\begin{proof}
 Letting $E = \partial\Omega$ and $T$ be the $\sigma$-algebra of all Borel sets on $\partial\Omega$, Lemma \ref{lem:cvae}  in the appendix proves that $L^1(\partial\Omega,(\mu_\theta)_{\theta\in\mathcal{S}})$ is a Banach space. We then observe that $L^2(\partial\Omega,(\mu_\theta)_{\theta\in\mathcal{S}}) \subset L^1(\partial\Omega,(\mu_\theta)_{\theta\in\mathcal{S}})$ since $\mu_\theta(\partial\Omega)$ is bounded independently of $\theta\in\mathcal{S}$. 
\end{proof}

\begin{definition}[Trace operator on $H^1(\Omega)$]\label{def:tracemono}
We define the domain $H_{\rm tr}^1(\Omega)$ of the trace operator as the set of all $u\in H^1(\Omega)$ such that there exists a function  $g \in L^2(\partial\Omega,(\mu_\theta)_{\theta\in\mathcal{S}})$ such that for all $\theta\in{\mathcal{S}}$  we have
 $\gamma_\theta u(z) = g(z)$ for $\mu_\theta$-a.e. $z\in\partial\Omega$. Then $g$ is unique and is called the trace of $u$ on $\partial\Omega$, denoted by ${\rm tr }(u)$.
\end{definition}

\begin{example}\label{exa:fisund}{\bf 1D case where $H_{\rm tr}^1(\Omega)\neq H^1(\Omega)$.}
 
Let us take the example $\Omega = ]0,1[\cup]1,2[$, and $u(x) = x$ for $x\in ]0,1[$, and  $u(x) = x-1$ for $x\in ]1,2[$. Then $\gamma_1 u(1) = 1$, $\gamma_1 u(2) = 1$, $\gamma_{\minus 1} u(0) = 0$, $\gamma_{\minus 1} u(1) = 0$. Hence in general, for $z\in \partial\Omega$,  $\gamma_\theta u(z)$ is not constant with respect to $\theta$.
\end{example}
\begin{example}\label{exa:fisdeuxd}{\bf 2D case where $H_{\rm tr}^1(\Omega)\neq H^1(\Omega)$.}

Let us consider $\Omega = (]0,1[\times ]-1,1[) \setminus (\{\frac 1 2\}\times[0,1])$. We define the function $u(x_1,x_2) = - x_2$ for $(x_1,x_2)\in ]0,\frac 1 2[\times ]0,1[ $, $u(x_1,x_2) = x_2$ for $(x_1,x_2)\in ]\frac 1 2,1[\times ]0,1[ $ and $u(x_1,x_2) = 0$ for $(x_1,x_2)\in ]0,1[\times ]-1,0[$. The domain is connected, contrarily to Example \ref{exa:fisund}, and $\gamma_{(1,0)} u( (\frac 1 2,s)) = -s$ and  $\gamma_{(-1,0)} u( (\frac 1 2,s)) = s$ for all $s\in]0,1[$.
\end{example}
 \begin{theorem} [$H_{\rm tr}^1(\Omega)$ is closed]\label{thm:huntclosed}
 Let $\Omega$ be any open bounded subset of $\mathbb{R}^d$ with $d\in\mathbb{N}\setminus\{0\}$. Then $H_{\rm tr}^1(\Omega)$ is closed, and therefore contains the closure $\widetilde{H}^1(\Omega)$ of $H^1(\Omega)\cap C^0(\overline{\Omega})$ in $H^1(\Omega)$. Moreover, the operator ${\rm tr }:H_{\rm tr}^1(\Omega)\to L^2(\partial\Omega,(\mu_\theta)_{\theta\in\mathcal{S}})$ is continuous and, for all $u,v\in  H_{\rm tr}^1(\Omega)$ and $\theta\in{\mathcal{S}}$, we have:
\begin{equation}\label{eq:intbypartshun}
 \int_\Omega (u(x)\partial_\theta v(x) +v(x)\partial_\theta u(x)){\rm d}x = \int_{\partial\Omega} \frac{{\rm tr }(u)(z){\rm tr }(v)(z) -{\rm tr }(u)(z_{\minus\theta}(z)){\rm tr }(v)(z_{\minus\theta}(z))}{|z-z_{\minus\theta}(z)|} {\rm d}\mu_\theta(z).
\end{equation}
 \end{theorem}
 \begin{proof}
  Recall that, using \eqref{eq:ineqtracezero} in Theorem \ref{thm:deftraceoned}, we have, for any $u\in H_{\rm tr}^1(\Omega)$ and any $\theta \in \mathcal{S}$,
  \[
\int_{\partial\Omega} |{\rm tr }(u)|^2 {\rm d}\mu_\theta = \int_{\partial\Omega} |\gamma_\theta u|^2 {\rm d}\mu_\theta \le 2 \max(1,{\rm diam}(\Omega)^2)\Vert u\Vert_{H^1(\Omega)}^2.
  \]
  We get that, for any Cauchy sequence $(u_n)_{n\in\mathbb{N}}$ of elements of $H_{\rm tr}^1(\Omega)$ with limit $u\in H^1(\Omega)$, the sequence $({\rm tr }(u_n))_{n\in\mathbb{N}}$ converges in $L^2(\partial\Omega,(\mu_\theta)_{\theta\in\mathcal{S}})$ which is a Banach space by Lemma \ref{lem:ldsbanach}. Denoting by $g$ its limit, we get that $g = \gamma_\theta u$ $\mu_\theta$-a.e. on $\partial\Omega$, which proves that $u\in H_{\rm tr}^1(\Omega)$.
  
  Then \eqref{eq:intbypartshun} is a consequence of Theorem \ref{thm:proptraceoned}.
 \end{proof}
 
\begin{figure}[!ht]
\begin{center}
\resizebox{.5\linewidth}{!}{
 \includegraphics{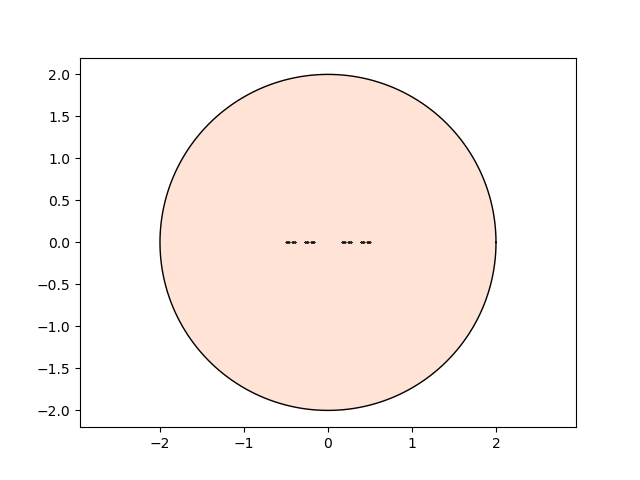}
 }
 \end{center}
 \caption{Domain $\Omega$ in Example \ref{exa:cantcirc}. \label{fig:cantcirc}}
\end{figure}

\begin{definition}[Functions with null trace]
We define $H_{\rm tr,0}^1(\Omega) =  \{ u \in H_{\rm tr}^1(\Omega)~\text{such that}~{\rm tr }(u) =0 \}.$
\end{definition}

\begin{remark} Note that $ H^1_0(\Omega)\subset H_{\rm tr,0}^1(\Omega)\subset H_{\rm tr}^1(\Omega)$ and that
 \[H_{\rm tr,0}^1(\Omega)  = \{ u \in H^1(\Omega)~\text{such that}~\gamma_\theta u=0~\text{for any}~\theta \in \mathcal{S} \}. \]
\end{remark}

\begin{example}\label{exa:cantcirc}{\bf Case where $H^1_0(\Omega)\neq H_{\rm tr,0}^1(\Omega)$}
In \cite[example 5.3]{are2024perron}, the authors consider the case $d=2$, they define the set $S = (\mathcal{C}_{\frac 1 3}-\frac 1 2)\times\{0\}$ (the 1D middle third Cantor set is then inserted in the 2D space) and they set $\Omega = B(0,2)\setminus S$, which implies that $\partial\Omega = S\cup \partial B(0,2)$ (see Figure \ref{fig:cantcirc}). Then they prove that the harmonic function $u\in H^1_{\rm loc}(\Omega)\cap C^0(\overline{\Omega})$ such that $u(z) = 1$ for all $z\in  S$ and $u(z) = 0$ for all $z\in\partial B(0,2)$ verifies $u\in H^1(\Omega)$. They also prove that $u\notin H^1_0(\Omega)$ and that the approximative trace of $u$ in the sense of Definition \ref{def:aptrace} is null, since the Hausdorff one-dimensional measure of $S$ is null. In this case, the notion of approximative trace and of trace in the sense of Definition \ref{def:tracemono} coincide. Indeed  Lemma \ref{lem:negl} proves that $\mu_\theta(S) = 0$ for any $\theta\in\mathcal{S}$ (note that $\rho = \frac 1 3$ is the limit case where $\lambda^1(\mathcal{C}_\rho)= \mathcal{H}^{1}(\mathcal{C}_\rho)=0$). This yields $\gamma_\theta u(z) = 0$ for $\mu_\theta$-a.e. $z\in \partial\Omega$ for all $\theta\in\mathcal{S}$, and therefore $u\in H_{\rm tr,0}^1(\Omega)$.

\end{example}
We then have the following lemma.
\begin{lemma}
For any $u\in  H^1(\Omega)$ such that there exists $\theta\in\mathcal{S}$ with $\gamma_{\theta} u(z) = 0$ for $\mu_\theta$-a.e. $z\in\partial\Omega$, then
\begin{equation}\label{eq:poinhun}
 \Vert u\Vert_{L^2(\Omega)} \le {\rm diam}(\Omega) \Vert \nabla u\Vert_{L^2(\Omega)}.
 \end{equation}
 Therefore the space $H_{\rm tr,0}^1(\Omega)$ is a Hilbert space, with the scalar product $\langle u,v\rangle := \int_\Omega \nabla u(x)\cdot\nabla v(x){\rm d}x$.
\end{lemma}
\begin{proof}
 For any  $\theta\in\mathcal{S}$, $|\partial_\theta u| \le | \nabla u|$ a.e. in $\Omega$. Then \eqref{eq:poinhun} is a consequence of \eqref{eq:pointheta} in Theorem \ref{thm:proptraceoned}.
 We then observe that $H_{\rm tr,0}^1(\Omega)\subset H_{\rm tr}^1(\Omega)$ is closed since ${\rm tr }:H_{\rm tr}^1(\Omega)\to L^2(\partial\Omega,(\mu_\theta)_{\theta\in\mathcal{S}})$ is continuous. The fact that $\langle u,v\rangle$ is a scalar product is a consequence of \eqref{eq:poinhun}.
\end{proof}

Before a few examples, we give a simple sufficient condition leading to $H_{\rm tr}^1(\Omega) = H^1(\Omega)$.

 \begin{lemma}\label{lem:sufcondmono} Let $(\Omega_i)_{i\in I}$ be a countable family of open subsets of $\Omega$, such that
 \begin{itemize}
  \item for any $i\in I$, $H_{\rm tr}^1(\Omega_i) = H^1(\Omega_i)$,
  \item   for all $\theta\in \mathcal{S}$, $\partial_\theta\Omega\setminus \bigcup_{i\in I} \partial_\theta\Omega_i$ is $\mu_\theta$-negligible,

  \item for all  $i\in I$, $\widetilde{\partial\Omega}\cap \widetilde{\partial\Omega_i}\cap \bigcup_{j\neq i} \widetilde{\partial\Omega_j}$ is $\mu_\theta$-negligible for all $\theta\in \mathcal{S}$ (see \eqref{eq:defdomegatilde} for the definition of $\widetilde{\partial\Omega}$). 

 \end{itemize}
Then $H_{\rm tr}^1(\Omega) = H^1(\Omega)$.  
 \end{lemma}
\begin{proof}
 Let $u\in H^1(\Omega)$ and let $u_i\in H^1(\Omega_i)$ be the restriction of $u$ to $\Omega_i$ for any $i\in I$. We denote by $g_i$ a representative of ${\rm tr }(u_i)$, the trace of $u_i$ on $\partial\Omega_i$, defined by Definition \ref{def:tracemono}. By Lemma \ref{lem:included}, we have that, for all $\theta\in \mathcal{S}$,  $g_i(z) = \gamma_\theta u_i(z) = \gamma_\theta u(z)$ for $\mu_{\theta}$-a.e. $z\in \partial_\theta\Omega_i\cap\partial_\theta\Omega$. 
 
 \medskip
 
 Let us define the function $g:\partial\Omega\to\mathbb{R}$ by $g(z) = g_i(z)$ for any $z\in \widetilde{\partial\Omega}\cap \widetilde{\partial\Omega_i}\setminus \bigcup_{j\neq i} \widetilde{\partial\Omega_j}$ and all $i\in I$, and $g(z) = 0$ for any other $z\in\partial\Omega$. We then have, for any $\theta\in\mathcal{S}$, $g(z) = g_i(z) = \gamma_\theta u(z)$ for $\mu_\theta$ a.e. $z\in \partial_\theta\Omega_i\cap\partial_\theta\Omega \subset \widetilde{\partial\Omega}\cap \widetilde{\partial\Omega_i}$. Since the complementary in $\partial_\theta\Omega$ of $\bigcup_{i\in I}(\partial_\theta\Omega_i\cap\partial_\theta\Omega)$ is $\mu_\theta$-negligible, we obtain that, for  all $\theta\in \mathcal{S}$, $g(z) = \gamma_\theta u(z)$ for  $\mu_{\theta}$-a.e. $z\in \partial_{\theta}\Omega$.   
 This implies that the equivalence class of $g$ in $L^2(\partial\Omega,(\mu_\theta)_{\theta\in\mathcal{S}})$ is the trace of $u$ in the sense of Definition  \ref{def:tracemono}.
\end{proof}

 \medskip
 
\begin{example}\label{exa:cantorund}
 Let us consider the example \ref{ex:cantor}, that is $\Omega = [0,1]\setminus \mathcal{C}_\rho$, with $\rho\in (0,\frac 1 3)$. This examples cannot be handled in the framework of the approximative trace of \cite{ae2011dirtoneu,sauter2020uniq} since the measure $\mathcal{H}^{d-1}$ is locally infinite everywhere on $\partial\Omega$.
In this example, we have $H_{\rm tr}^1(\Omega)= H^1(\Omega)$ since there is no $m, m'\in\mathbb{N}$ such that $c_m = d_{m'}$, so the supports of $\gamma_1$ and $\gamma_{\minus 1}$ are disjoint, which implies that any $u\in H^1(\Omega)$ has a trace, applying Lemma \ref{lem:sufcondmono}.
\end{example}

\medskip

\begin{figure}[!ht]
\resizebox{\linewidth}{!}{
 \includegraphics{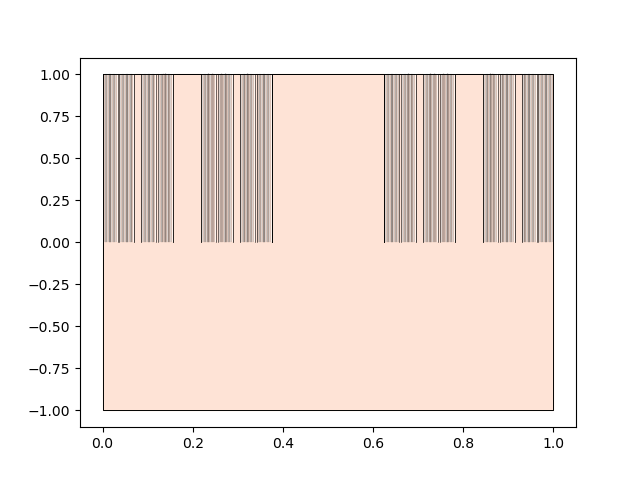}
 \includegraphics{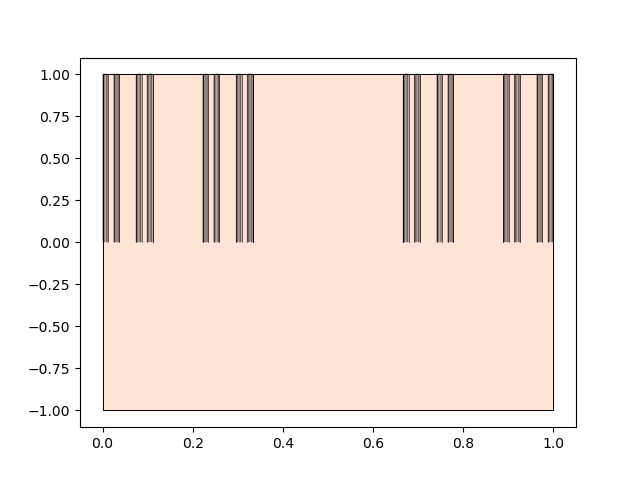}
 }
 \caption{Example \ref{exa:bicantor}. Left: $\rho = \frac 1 4$. Right: $\rho = \frac 1 3$. \label{fig:bicantor}}
\end{figure}

\begin{example}\label{exa:bicantor}
Let us now consider a two-dimensional domain which is connected, but which cannot have an approximative trace in the sense of \cite{are2024perron}.
The domain is depicted in Figure \ref{fig:bicantor}. We denote, for a given $\rho\in (0,\frac 1 3]$, by
\[
 \Omega = \Big( ([0,1]\setminus \mathcal{C}_\rho )\times ]-1,1[ \Big)\cup \Big(]0,1[\times ]-1,0[\Big).
\]
Then the boundary of $\Omega$ contains $\mathcal{C}_\rho \times [0,1[$, which has infinite local Hausdorff measure $\mathcal{H}^{d-1}$ for $\rho < \frac 1 3$, since it contains the set  $\mathcal{C}_\rho \times [0,1]$. But the set $\widetilde{\partial\Omega} $ is given by the union of all the following sets:
\begin{itemize}
 \item the boundary of $]0,1[\times ]-1,1[$,
 \item the sets of the form $\{c_m\}\times [0,1]$ and $\{d_m\}\times [0,1]$,
 \item the set $\mathcal{C}_\rho \times \{0\}$.
\end{itemize}
Then $\Omega$ is connected and $H_{\rm tr}^1(\Omega)= H^1(\Omega)$, applying Lemma \ref{lem:sufcondmono} with $\Omega_0 = ]0,1[\times ]-1,0[$ and $\Omega_m = ]c_m,d_m[\times ]0,1[$.

\end{example}

\medskip

\begin{figure}[!ht]
\begin{center}
\resizebox{.5\linewidth}{!}{
 \includegraphics{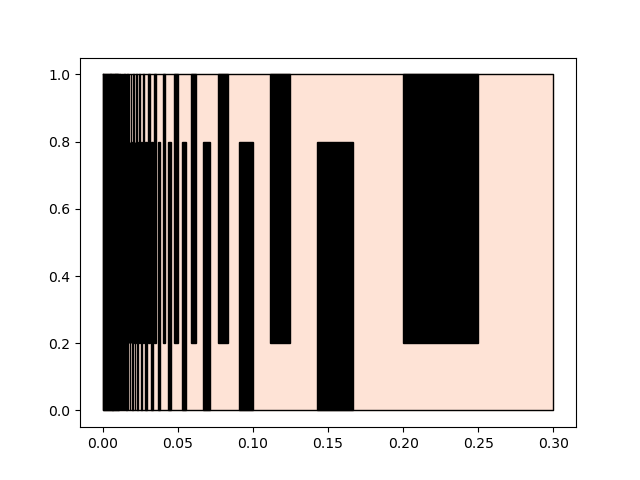}
 }
 \end{center}
 \caption{Domain $\Omega$ in Example \ref{exa:serpent}. \label{fig:serpent}}
\end{figure}

\begin{example}\label{exa:serpent}
We consider an example similar to \cite[Figure 2]{sauter2020uniq}. We define $\Omega\subset\mathbb{R}^2$ (depicted in the left part of Figure \ref{fig:serpent}) defined by
\[
 \Omega = ]0,\frac 3 {10}[\times]0,1[\setminus \bigcup_{k\in\mathbb{N}^\star} \Big([\frac 1 {4k+3},\frac 1 {4k+2}]\times[0,\frac 4 5]\cup[\frac 1 {4k+1},\frac 1 {4k}]\times[\frac 1 5,1]\Big).
\]
We notice that 
\[
 \{0\}\times[0,1]=\partial\Omega\setminus\widetilde{\partial\Omega}.
\]

Then $\Omega$ has density $0$ at $\{0\}\times[0,1]$ in the sense given in \cite{sauter2020uniq}, and the approximative trace defined in \cite{sauter2020uniq} is unique. In this case we also have $H_{\rm tr}^1(\Omega)= H^1(\Omega)$ since Lemma \ref{lem:sufcondmono}  applies, letting $\Omega_k = ]\frac 1 {4(k+1)}, \frac 1 {4k}[\times]0,1[\cap\Omega$.
\end{example}

\appendix

\section{Regularity along lines in $W_\theta(\Omega)$}\label{ap:hund}

In this section, for the simplicity of the notation, we consider the case where $\theta$ is the first vector of the canonical basis of $\mathbb{R}^d$, and we denote in this section by $\partial_1   \varphi$, for any $\varphi\in C^\infty_c(\Omega)$, the partial derivative of $\varphi$ with respect to its first argument.  For any $u\in L^2(\Omega)$, we say that $\partial_1   u\in L^2(\Omega)$ if there exists $v\in L^2(\Omega)$ such that
 \[
  \forall \varphi\in C^\infty_c(\Omega),\ \int_\Omega u(x) \partial_1   \varphi(x) {\rm d}x =  -\int_\Omega v(x) \varphi(x) {\rm d}x.
 \]
 In this case, $v$ is uniquely defined, and we denote by $\partial_1   u = v$.  We then define 
 \[
W_1(\Omega) = \{ u \in L^2(\Omega); \partial_1   u\in  L^2(\Omega)\}.
 \]
 Then $W_1(\Omega)$ is a Hilbert space, with the scalar product
 \[
  \langle u,v\rangle_1 = \int_\Omega (u(x) v(x) + \partial_1   u(x)\partial_1   v(x)){\rm d}x.
 \]
 For any $x = (s,y)\in \mathbb{R}^d$, we denote $\mathcal{P}_1(x) = y$, and we define the following sets:
 \begin{equation*}
  \mathcal{P}_1(\Omega)= \{ y\in \mathbb{R}^{d-1};\exists s\in\mathbb{R}, (s,y)\in\Omega \}~\text{and}~
   \forall y\in\mathcal{P}_1(\Omega),\ \omega_1(y) = \{ s\in \mathbb{R}~\text{such that}~ (s,y) \in \Omega\}\subset \mathbb{R}.
 \end{equation*}
  
We have the following lemma.
\begin{lemma}\label{lem:underivative}
 Let $u\in W_1(\Omega)$. Then, for any representative of $u$ and any representative of $\partial_1  u$, again denoted by $u$ and $\partial_1  u$, there exists a measurable set $A_1\subset \mathcal{P}_1(\Omega)$ whose complementary in  $\mathcal{P}_1(\Omega)$ is $\lambda^{d-1}$-negligible such that, for all $y \in A_1$, the function defined for all $s\in \omega_1(y)$  by $f(s) = u(s,y)$ is such that $f\in H^1(\omega_1(y))$, with, for $\lambda$-a.e. $s\in\omega_1(y)$ (for the Lebesgue measure of $\mathbb{R}$),
$D\!f(s) = \partial_1  u(s,y)$, where we denote by $D\!f$ the weak derivative of $f$ in the open set $\omega_1(y)$.
\end{lemma}
\begin{proof}
The representatives $u$ and $\partial_1  u$ are such that $u\in {\mathcal L}^2(\Omega )$ and  $\partial_1  u\in {\mathcal L}^2(\Omega )$. The lemma will be proved by showing that, for a.e. $y \in \mathcal{P}_1(\Omega)$, for all $i \in I(y)$, the functions $f:=u(\cdot,y)$ and $g:=\partial_1 u(\cdot,y) $ are such that
\begin{align}
& f \in L^2(\omega_1(y)),\label{eq:itemun}
\\
& g\in L^2(\omega_1(y)),\label{eq:itemdeux}
\\
&D\!f = g,\label{eq:itemtrois}
\end{align}
 We could as well write, for given $y$, $g=\partial_1 u(\cdot,y)$ and $D\!f=D(u(\cdot,y))$ (notice the role of the parentheses). 

\medskip

{\bf Step 1}

\smallskip

Applying Fubini-Tonelli's theorem gives, with extending $u$ by $0$ outside $\Omega $,
\[
\int_{{\mathbb R}^{d-1}} \big( \int_{\mathbb R} u(x,y)^2 {\rm d}x \big) {\rm d}y =\int_{\Omega } u(z)^2 {\rm d}z < +\infty.
\]
A more precise consequence of Fubini-Tonelli's theorem is that the mapping $y \mapsto \int_{\mathbb R} u(x,y)^2 {\rm d}x$ is a measurable mapping from ${\mathbb R}^{d-1}$ to ${\mathbb R}$ (in the sense that it is a.e. equal to a measurable function) and integrable. We therefore have, for a.e.  $y\in\mathcal{P}_1(\Omega)$, $u(\cdot,y) \in L^2(\omega_1(y))$. This concludes the proof of \eqref{eq:itemun}.

The same reasoning, applied to  $\partial_1 u$, provides the proof of \eqref{eq:itemdeux}.

\medskip

{\bf Step 2} Proof of \eqref{eq:itemtrois} by regularisation.

\smallskip

Let us denote by $\mathcal{B}$ the set of all open sets under the form $\prod_{k=1}^d  ]\alpha_k,\beta_k[$, with $\alpha_k,\beta_k\in{\mathbb Q}$, such that
\[
 \forall K\in \mathcal{B},\ K\subset \Omega .
\]
We have that $\mathcal{B}$ is countable, and that
\[
 \Omega  = \bigcup_{K\in \mathcal{B}} K.
\]
For a given  $K = \prod_{k=1}^d  ]\alpha_k,\beta_k[\in \mathcal{B}$, let us first prove \eqref{eq:itemtrois} on $]\alpha_1,\beta_1[$ for a.e. $y\in K_1:= \prod_{k=2}^d  ]\alpha_k,\beta_k[$.

Let $A_1$ be the set of all $y\in K_1$ be such that \eqref{eq:itemun} and \eqref{eq:itemdeux} hold. 

\medskip

Let us define the mollifying function $\rho \in C^\infty_c({\mathbb R},{\mathbb R})$, $\rho \ge 0$, $\rho(x)=0$ if $|x|>1$, $\int_{\mathbb R} \rho(x) {\rm d}x=1$.

For any ${n \in {\mathbb N}^\star}$, let us define $\rho_n(x)=\frac 1 n \rho(nx)$ ($x \in {\mathbb R}$).

\smallskip

For all  $y \in A_1$, ${n \in {\mathbb N}^\star}$, we define $u_{n}(x,y)=u(x, y) \star \rho_n := \int_{\alpha_1}^{\beta_1} u(s,y) \rho_n(x-s) {\rm d}s$. Hence $u_{n}(\cdot,y)  \in C^\infty(]\alpha_1,\beta_1[,{\mathbb R})$ and $u_{n}^\prime(x,y)=\int_{\alpha_1}^{\beta_1} u(s,y) \rho_n^\prime(x-s) {\rm d}s$.  For any $y \not \in A_1$, we set $u(\cdot,y)=0$.

\medskip

Let ${m \in {\mathbb N}^\star}$ be such that $1/m <  (\beta_1-\alpha_1)/2$. We set $I_m= ]\alpha_1+\frac 1 m,\beta_1 - \frac 1 m[$.

\smallskip

Let $\psi \in C^\infty_c(K_1,{\mathbb R})$. We then have, for any $x\in I_m$ and any $n>2m$, support$(\varphi) \subset K$, letting $\varphi(s,y) = \rho_n(x-s) \psi(y)$. We can then write,
since $\partial_1 $ is the weak derivative with respect to the first coordinate,
\begin{equation}
\int_{K_1} u_{n}^\prime(x,y) \psi(y) {\rm d}y=\int_{K} u(s,y) \rho_n^\prime(x-s) \psi(y){\rm d}s{\rm d}y= \int_{K} \partial_1 u( s  ,y) \rho_n(x- s ) \psi(y){\rm d}s{\rm d}y,
\label{ega-fon}
\end{equation}
since $\partial_1 \varphi(s,y) = - \rho_n^\prime(x-s) \psi(y)$.

\medskip

We define $v$ on $K$ by letting, for any $y\in A_1$ and $x\in]\alpha_1,\beta_1[$,   $v(x,y)=\int_{\alpha_1}^x \partial_1 u(s,y) {\rm d}s$. Hence we have $v(\cdot,y) \in C([\alpha_1,\beta_1],{\mathbb R})$ and
$Dv=\partial_1  u(\cdot,y)$ on $]\alpha_1,\beta_1[$ \cite{mazya}, and for any $y\in K_1\setminus A_1$, $v(\cdot,y)=0$.

\medskip

We repeat the regularisation method on $v$: we define, for any ${n \in {\mathbb N}^\star}$, $v_{n}(\cdot,y)=v(\cdot, y) \star \rho_n := \int_{\alpha_1}^{\beta_1} v(s,y) \rho_n(x-s) {\rm d}s$.

\smallskip

Let  $y \in A_1$, ${m \in {\mathbb N}^\star}$  such that $1/m <  (\beta_1-\alpha_1)/2$, and ${n \in {\mathbb N}^\star}$ with $n>2m$. For $x \in I_m$, $v_{n}^\prime(x,y)=\int_{\alpha_1}^{\beta_1} v( s  ,y) \rho_n^\prime(x- s ) {\rm d}s = \int_{\alpha_1}^{\beta_1} Dv( s  ,y) \rho_n(x- s ) {\rm d}s$.

\smallskip

Since $Dv=\partial_1 u(\cdot,y)$ on $]\alpha_1,\beta_1[$,
\[
v_{n}^\prime(x,y)=\int_{\alpha_1}^{\beta_1} \partial_1 u( s  ,y) \rho_n(x- s ) {\rm d}s .
\]
Let $\psi \in C^\infty_c(K_1,{\mathbb R})$. We have
\[
\int_{K_1} v_{n}^\prime(x,y) \psi(y) {\rm d}y=\int_{K} \partial_1 u( s  ,y) \rho_n(x-s) \psi(y){\rm d}s{\rm d}y.
\]
Using \eqref{ega-fon}, we get
\[
\int_{K_1} (u_{n}-v_{n})^\prime(x,y) \psi(y) {\rm d}y=0.
\]
Now only considering $x \in {\mathbb Q}\cap I_m$, using the countability of  ${\mathbb Q}$ and ${\mathbb N}$, we get that for a.e. $y\in K_1$, $m$ with $1/m <  (\beta_1-\alpha_1)/2$ and $n>2m$,
\[
(u_{n}-v_{n})^\prime(x,y)=0.
\]
Therefore, since $u_{n}(\cdot,y)$ and $v_{n}(\cdot,y)$ are elements of $C^\infty(]\alpha_1,\beta_1[)$,
$(u_{n}-v_{n})^\prime(x,y)=0$ on $I_m$. Hence  $(u_{n}-v_{n})(\cdot,y)$ is constant on  $I_m$. Letting $n\to\infty$, we get that the function $(u-v)(\cdot,y)\in L^2(]\alpha_1,\beta_1[)$ is constant on $I_m$. Letting  $m \to +\infty$, we get that $(u-v)(\cdot,y)\in L^2(]\alpha_1,\beta_1[)$ is constant on $]\alpha_1,\beta_1[$.

\medskip

We therefore obtain that, for a.e. $y\in K_1$, $u(\cdot,y) \in H^1(]\alpha_1,\beta_1)[$ and $D(u(\cdot,y))=\partial_1 u(\cdot,y)$.

Since $\mathcal B$ is countable, for a.e. $y\in \mathcal{P}_1(\Omega)$,  the functions $u(\cdot,y) \in L^2(\omega_1(y))$, $\partial_1 u(\cdot,y) \in L^2(\omega_1(y))$ are such that, for all $K\in \mathcal B$ such that $]\alpha_1^K,\beta_1^K[\subset \omega_1(y)$,
$u(\cdot,y) \in H^1(]\alpha_1^K,\beta_1^K[)$ and $D(u(\cdot,y))=\partial_1 u(\cdot,y)$.

This yields $u(\cdot,y) \in H^1(\omega_1(y))$ and $D(u(\cdot,y))=\partial_1 u(\cdot,y)$ on $\omega_1(y)$. 

\end{proof}

\begin{remark}We could as well notice that the mapping $y \mapsto u(\cdot,y)$ is measurable from ${\mathbb R}^{d-1}$ to $L^2({\mathbb R})$, owing to the separability of $L^2({\mathbb R})$. But this property is not used in this proof.
\end{remark}

\begin{lemma}\label{lem:fhun}
 Let $\alpha<\beta$ be two reals and let $f\in H^1(]\alpha,\beta[)\cap C^0([\alpha,\beta])$ be given. Then the following properties hold:
 \begin{equation}\label{eq:valbeta}
   f(\beta)=\frac 1 {\beta - \alpha}\int_\alpha^\beta \big(f(t) + D f(t)(t-\alpha)\big){\rm d}t,
 \end{equation}
 \begin{equation}\label{eq:majfhun}
   \max(f(\beta)^2, f(\alpha)^2 )\le \Vert f\Vert_{H^1(]\alpha,\beta[)}^2 \frac{2\max(1, (\beta - \alpha)^2)}{\beta - \alpha},
 \end{equation}
\begin{equation}\label{eq:majsumfhun}
   (f(\beta)+f(\alpha))^2 \le \Vert f\Vert_{H^1(]\alpha,\beta[)}^2 \frac{4\max(1, (\beta - \alpha)^2)}{\beta - \alpha}
 \end{equation}
 and
\begin{equation}\label{eq:majdiffhun}
\Big(\frac{f(\beta) - f(\alpha)} {\beta - \alpha}\Big)^2 \le \Vert f\Vert_{H^1(]\alpha,\beta[)}^2\frac{1}{\beta - \alpha}.
 \end{equation}
Moreover, if $f(\beta)=0$, then
\begin{equation}\label{eq:poinfhun}
\Vert f\Vert_{L^2(]\alpha,\beta[)} \le (\beta - \alpha) \Vert D f\Vert_{L^2(]\alpha,\beta[)}.
 \end{equation}
\end{lemma}
\begin{proof}
 If $f\in C^1([\alpha,\beta]$, we can write
  $f(s) = f(t) + \int_t^s D f(w){\rm d}w$. Introducing the function $\Psi(s,t) = t-\alpha$ for $t< s$ and $\Psi(s,t) = t-\beta$ for $t> s$, we can write
\[
 \forall s\in[\alpha,\beta],\ f(s) = \frac 1 {\beta - \alpha}\int_\alpha^\beta (f(t) + D f(t)\Psi(s,t)){\rm d}t.
\]
It yields, applying the Cauchy-Schwarz inequality,
\[
 \forall s\in[\alpha,\beta],\ f(s)^2 \le \frac 2 {\beta - \alpha}\int_\alpha^\beta (f(t)^2 + D f(t)^2(\beta - \alpha)^2){\rm d}t,
\]
and therefore we get \eqref{eq:majfhun}. We therefore deduce \eqref{eq:majsumfhun}. Finally, writing
\[
 f(\beta) - f(\alpha) = \int_\alpha^\beta  D f(t){\rm d}t,
\]
we get \eqref{eq:majdiffhun}.
Finally, if $f(\beta) = 0$, we have for any $s\in[\alpha,\beta]$,
\[
 f(s) = -\int_{]s,\beta[} Df(t){\rm d}t,
\]
which yields
\[
 f(s)^2 \le (\beta - s)\int_{]s,\beta[} Df(t)^2{\rm d}t \le  (\beta - \alpha) \Vert D f\Vert_{L^2(]\alpha,\beta[)}^2.
\]
Integrating the preceding inequality with respect to $s\in[\alpha,\beta]$ provides \eqref{eq:poinfhun}.
\end{proof}

\section{Bounded integrability with respect to a family of measures}

  Let $(E,T)$ be a measurable set, where we denote by $T$ a given $\sigma$-algebra on the set $E$. Let  $(m_\theta)_{\theta\in\mathcal{S}}$ be a family of measures on $(E,T)$. We define $\mathcal{L}^1(E,T,(m_\theta)_{\theta\in\mathcal{S}})$ as 
\begin{multline*}
 \mathcal{L}^1(E,T,(m_\theta)_{\theta\in\mathcal{S}}) =\{f:E\to\mathbb{R},\exists M\ge 0,\ \forall\theta\in \mathcal{S},\\
 \exists g\in  \mathcal{L}^1(E,T,m_\theta), f = g~ m_\theta-\hbox{ a.e. and }\int_E |g| {\rm d}m_\theta\le M\}.
\end{multline*}
We then say that two elements $f,g\in \mathcal{L}^1(E,T,(m_\theta)_{\theta\in\mathcal{S}})$ are equivalent if $f = g$ $m_\theta$-a.e. for all $\theta\in\mathcal{S}$, and we define the set  $L^1(E,T,(m_\theta)_{\theta\in\mathcal{S}})$ of all the equivalence classes of the elements of $\mathcal{L}^1(E,T,(m_\theta)_{\theta\in\mathcal{S}})$. 

Then, for any $f\in L^1(E,T,(m_\theta)_{\theta\in\mathcal{S}})$ and $\theta\in\mathcal{S}$, we define
$ \int_E f {\rm d}m_\theta := \int_E g {\rm d}m_\theta$, where $g\in \mathcal{L}^1(E,T,m_\theta)$ is any representative of $f$. We define a norm on $L^1(E,T,(m_\theta)_{\theta\in\mathcal{S}})$ by
\[
 \forall f\in L^1(E,T,(m_\theta)_{\theta\in\mathcal{S}}),\ \Vert f\Vert_{L^1(E,T,(m_\theta)_{\theta\in\mathcal{S}})} := \sup_{\theta\in\mathcal{S}}\int_E |f| {\rm d}m_\theta.
\]
\begin{remark}
 In the preceding definition, we do not require that the elements of $\mathcal{L}^1(E,T,(m_\theta)_{\theta\in\mathcal{S}})$ be measurable, but only equal $m_\theta$-a.e. to a measurable function for all $\theta\in\mathcal{S}$.
\end{remark}
\begin{lemma}\label{lem:cvae}
The space $L^1(E,T,(m_\theta)_{\theta\in\mathcal{S}})$ is a Banach space.
\end{lemma}

\begin{proof}

We follow the proofs of \cite[Th\'eor\`emes 4.47 and 4.49]{gh}.

\medskip

Let $(f_n)_{n\in\mathbb{N}}$ be a Cauchy sequence in $L^1(E,T,(m_\theta)_{\theta\in\mathcal{S}})$, and let us also denote by $f_n\in \mathcal{L}^1(E,T,(m_\theta)_{\theta\in\mathcal{S}})$ a representative of $f_n$ for any $n\in\mathbb{N}$.  

\medskip

\textbf{Step 1.} For any $n\in\mathbb{N}$, we choose $\varphi(n)\in\mathbb{N}$ such that, for any $p,q\ge \varphi(n)$, $ \sup_{\theta\in\mathcal{S}}\int_E|f_p - f_q| {\rm d}m_\theta\le 2^{-n}$ and $\varphi(n)>\varphi(n-1)$ if $n\ge 1$. 

\medskip

\textbf{Step 2.} We define the series of functions $(g_n)_{n\in\mathbb{N}}$ with general term $g_n := f_{\varphi(n+1)}-f_{\varphi(n)}$. From the definition of the subsequence $\varphi$, we have $ \sup_{\theta\in\mathcal{S}}\int |g_n| {\rm d}m_\theta\le 2^{-n}$ for any $n \ge 0$. We define the measurable function $G : E \to \mathbb{R}_+ \cup \{ + \infty \}$ by
\[
 \forall x\in E,\ G(x) = \sum_{n\in\mathbb{N}} |g_n(x)| \in [0,+\infty].
\]
For any $\theta\in\mathcal{S}$ we have
\[
\int_E G {\rm d}m_\theta =  \sum_{ n \ge 0} \int_E |g_n|{\rm d}m_\theta  \le \sum_{n\ge 0} 2^{-n} = 2.
\]
This implies that $m_\theta(E \setminus A)=0$ for any $\theta \in \mathcal{S}$ where $A= \{ x \in E~\text{such that}~G(x) < + \infty \}$.
We replace $G(x)=+\infty$ by $G(x)=0$ for any $x\notin A$. In particular for any $x \in A$ we have that
the real series with general term $g_n(x)$ is absolutely convergent. We define the measurable function $f : E \to \mathbb{R}$ by
\[
 \forall x\in A, \ f(x) = f_{\varphi(0)}(x) +  \sum_{n\in\mathbb{N}}g_n(x)~\text{and}~f(x)=0~\text{for any}~x\notin A.
\]

\medskip

\textbf{Step 3.} Since, for any $N \ge 1$,
\[
 \forall x\in A, \ f_{\varphi(N+1)}(x) = f_{\varphi(0)}(x) +  \sum_{n=0}^{N} g_n(x),
\]
 we get that, for all $\theta\in\mathcal{S}$ and all $x\in A$,  $f_{\varphi(N)}(x)\to f(x)$  as $N\to+\infty$ and, for $m_\theta$-a.e. $x\in E$, it holds  $|f_{\varphi(N)}(x)|\le |f_{\varphi(0)}(x)| + G(x)$. By the dominated convergence theorem, we get that $f\in {\mathcal{L}}^1(E,T,m_\theta)$ (in the sense that there exists $g_\theta\in {\mathcal{L}}^1(E,T,m_\theta)$ such that $f = g_\theta$ $\mu_\theta$-a.e.) and $\int_E | f - f_{\varphi(N)}|{\rm d}m_\theta\to 0$ as $N\to +\infty$. This implies that $\int | f|{\rm d}m_\theta$ is bounded independently of $\theta\in\mathcal{S}$, and therefore $f\in \mathcal{L}^1(E,T,(m_\theta)_{\theta\in\mathcal{S}})$.

\medskip

\textbf{Step 4.} Let $n\in\mathbb{N}$ such that $n\ge \varphi(0)$ and let $N_n\in\mathbb{N}$ be such that $\varphi(N_n)\le n <\varphi(N_n+1)$. Then, for any $\theta\in\mathcal{S}$,
\[
 \int_E | f - f_{n}|{\rm d}m_\theta \le \int_E | f - f_{\varphi(N_n)}|{\rm d}m_\theta + \int_E| f_{\varphi(N_n)} - f_n|{\rm d}m_\theta \le \sum_{k\ge N_n+1}\int_E |g_k|{\rm d}m_\theta + 2^{-N_n} \le  2^{-N_n+1}.
\]
This proves that
\[
 \lim_{n\to\infty}  \Vert f - f_n\Vert_{L^1(E,T,(m_\theta)_{\theta\in\mathcal{S}})} = 0,
\]
and concludes the proof of the lemma.
\end{proof}

\medskip

{\bf Acknowledgments.} The authors are deeply thankful to Wolfgang Arendt and Lucas Oger, for insightful discussions and suggestions.

\end{document}